\documentclass{sig-alternate}
\usepackage{amsmath,amsfonts}
\usepackage{amssymb}

\usepackage{epsfig}

\usepackage{graphics}
\usepackage{graphicx}

\usepackage{latexsym}
\usepackage{graphics}
\usepackage{amsmath}
\usepackage{xspace}
\usepackage{amssymb}
\usepackage{color}
\usepackage{cite}
\usepackage{latexsym}
\usepackage{graphics}
\usepackage{amsmath}
\usepackage{xspace}
\usepackage{amssymb}
\usepackage{epsfig}

\newcommand{\beql}[1]{\begin{equation}\label{#1}}
\newcommand{\eeql}{\end{equation}}
\newcommand{\eqn}[1]{(\ref{#1})}

\newcommand{\R}{\mathbb{R}}
\newcommand{\pr}{\mathbb{P}}
\newcommand{\E}{\mathbb{E}}

\newcommand{\ci}{{\cal I}}
\newcommand{\cs}{{\cal S}}
\newcommand{\ck}{{\cal K}}
\newcommand{\cx}{{\cal X}}
\newcommand{\cm}{{\cal M}}
\newcommand{\cq}{{\cal Q}}

\newcommand{\Z}{\mathbb{Z}}

\newtheorem{thm}{Theorem}
\newtheorem{lem}[thm]{Lemma}
\newtheorem{prop}[thm]{Proposition}
\newtheorem{cor}[thm]{Corollary}
\newtheorem{assumption}[thm]{Assumption}
\newtheorem{definition}[thm]{Definition}

\usepackage{soul}

\newcommand{\hull}[1]{\langle{#1}\rangle}

\newcommand{\bk}{\boldsymbol{k}}
\newcommand{\be}{\boldsymbol{e}}
\newcommand{\bZero}{\boldsymbol{0}}
\newcommand{\bOne}{\boldsymbol{1}}
\newcommand{\bX}{\boldsymbol{X}}
\newcommand{\bx}{\boldsymbol{x}}
\newcommand{\bU}{\boldsymbol{U}}
\newcommand{\bu}{\boldsymbol{u}}
\newcommand{\bv}{\boldsymbol{v}}
\newcommand{\boldeta}{\boldsymbol{\eta}}
\newcommand{\bxi}{\boldsymbol{\xi}}
\newcommand{\veps}{\varepsilon}
\newcommand{\wt}{\widetilde}
\newcommand{\N}{\mathbb{N}}
\newcommand{\cl}{\mathcal{L}}
\newcommand{\ct}{\mathcal{T}}
\def\red#1{{\color{red}#1}}

\def\old#1{}

\begin{document}

\title{A large-scale service system with packing constraints:\\ Minimizing the number of occupied servers\\
}


\author
{
Alexander L. Stolyar \\
Bell Labs, Alcatel-Lucent\\
600 Mountain Ave., 2C-322\\
Murray Hill, NJ 07974 \\
\texttt{stolyar@research.bell-labs.com}
\and
Yuan Zhong \\
University of California\\
465 Soda Hall, MC-1776\\
Berkeley, CA 94720\\
\texttt{zhyu4118@berkeley.edu}
}

\maketitle

\begin{abstract}
{We consider a large-scale service system model proposed in \cite{St2012}, which is}
motivated by the problem of efficient {placement} of virtual machines to physical host machines in 
a network cloud, so that the total number of occupied hosts is minimized. 
Customers of different types arrive to a system with an infinite number of servers. 
A server packing {\em configuration} is the vector $\bk = \{k_i\}$, 
where $k_i$ is the number of type-$i$ customers that the server ``contains''. 
Packing constraints are described by a fixed finite set of allowed configurations. 
Upon arrival, each customer is placed into a server immediately, 
subject to the packing constraints; 
the server can be idle or already serving other customers. 
After service completion, each customer leaves its server and the system.

It was shown in \cite{St2012} that a simple real-time algorithm, 
called {\em Greedy}, is asymptotically optimal in the sense of minimizing
$\sum_{\bk} X_{\bk}^{1+\alpha}$ in {the} stationary regime,
as the customer arrival rates grow to infinity. 
{(Here $\alpha > 0$, and $X_{\bk}$ denotes the number of servers with configuration $\bk$.)} 
In particular, 
when parameter $\alpha$ is small,
{\em Greedy} 
approximately solves the problem of minimizing $\sum_{\bk} X_{\bk}$, the number of occupied hosts.
In this paper we introduce the algorithm
called {\em Greedy with sublinear Safety Stocks (GSS)}, 
and show that 
it asymptotically solves the exact problem of minimizing $\sum_{\bk} X_{\bk}$. 
An important feature of the algorithm is that sublinear safety stocks of $X_{\bk}$ are created automatically -- when and where necessary --  
without having to determine {\em a priori} where they are required.
{Moreover, we also provide a tight characterization of the rate of convergence to optimality under {\em GSS}.}
The {\em GSS} algorithm is as simple as {\em Greedy}, and uses no more system state information than {\em Greedy} does.


\end{abstract}

\category{}{Network Services}{Cloud Computing}
\category{}{Probability and Statistics}{Markov Processes, Queueing Theory, Stochastic Processes}
\category{}{Design and Analysis of Algorithms}{Approximation Algorithms Analysis, Packing and Covering Problems}

\terms{Algorithms, Performance, Theory}


\section{Introduction}
We consider a service system model \cite{St2012} motivated by 
the problem of efficient placement of virtual machines (VMs) 
to physical host machines (servers) in a data center (DC)  \cite{Gulati2012}. 
A {\em service policy} decides to which server each 
incoming VM will be placed. 
We are interested in service policies 
that minimize the total number of occupied servers in the system. 
It is further desirable that the policy be simple, so that 
placement decisions are made in real time, and depend only 
on the current system state, but not on system parameters. 

Consider the following description of a DC. 
It consists of a number of servers. While servers may potentially 
have different characteristics, 
in this paper we assume that they are all the same. More specifically, 
let there be $N$ different types of resources (for example, type-$1$ resource 
can be CPU, type-$2$ resource can be memory, etc). 
For each $n \in \{1, 2, \ldots, N\}$, a server possesses amount $B_n > 0$ 
of type-$n$ resource. $I$ types of VMs arrive 
in a probabilistic fashion, and request services at the DC. 
Arriving VMs will be placed 
into the servers, occupying certain resources. 
More specifically, for $i \in \{1, 2, \ldots, I\}$, 
a type-$i$ VM requires 
amount $b_{i, n} > 0$ of type-$n$ resource during service, 
where $n \in \{1, 2, \ldots, N\}$. 
Once a VM completes its service, 
it departs the system, freeing up corresponding resources. 
We assume that service times of different VMs are independent. 

For each $i \in \{1, 2, \ldots, I\}$, 
let $k_i$ be the number of type-$i$ VMs that a server contains. 
Then the following {\em vector packing constraints} must be observed at all times. Namely, 
a server can contain $k_i$ type-$i$ VMs ($i \in \{1, 2, \ldots, I\}$) simultaneously 
if and only if 
\begin{equation}\label{eq:vec-packing}
\sum_{i} k_i b_{i,n} \le B_n,
\end{equation}
for each $n \in \{1, 2, \ldots, N\}$. In this case, the vector $\bk = (k_1, \ldots, k_I)$ 
is called a {\em server configuration}.

The model considered in this paper is similar to the DC described above, 
but different in the following two aspects. 
\begin{itemize}
\item[1.] While vector packing constraints (cf. Eq. \eqref{eq:vec-packing}) 
arise naturally in the context of VM placement, 
we make the more general assumption of 
so-called {\em monotone} packing constraints 
(cf. Section \ref{sec:model}) in our model. 
\item[2.] We consider a system with an infinite number of servers, 
where incoming VMs will be immediately placed into a server. 
For large-scale DCs, the number of servers is not a bottleneck, 
hence an infinite-server system reasonably approximates such DCs. 
\end{itemize}
We would also like to remark that 
an important assumption of our model is that the service requirement of a VM 
is not affected by potentially other VMs occupying the same server. 
This is a reasonable modeling assumption for multi-core servers, for example.

There can be different performance objectives of interest. 
For example, we may be interested in minimizing the total energy consumption \cite{Gulati2012}, 
or maximizing system throughput \cite{Maguluri2012}. 
In this paper, we are interested in minimizing the total number of 
occupied servers. These objectives are different but related. 
For example, by switching off idle servers, or keeping them in stand-by mode, 
we can reduce energy consumption by minimizing the number of occupied servers. 

In the main results of the paper, we introduce the policy called 
{\em Greedy with sublinear Safety Stocks ({\em GSS})}, 
and show that it asymptotically minimizes the total number of occupied servers in steady state, 
as the input flow rates of VMs grow to infinity. 
{{\em GSS} is a simple policy 
that makes placement decisions in real time, 
and based only on the current system state.}
Informally speaking, {\em GSS} places incoming VMs in a way that greedily minimizes 
a Lyapunov function, which asymptotically coincides with 
the total number of occupied servers. 
{\em GSS}  maintains non-empty safety stocks 
at every server configuration $\bk$ whenever $X_{\bk}$ becomes ``too small'',
so as to allow flexibility on VM placement. 
In other words, under {\em GSS}, there is a non-zero number of servers of every configuration, 
so that an incoming VM can potentially be placed into a server with any configuration. 
These safety stocks correspond to the discrepancy between the Lyapunov function 
and the total number of occupied servers, 
and grow ``sublinearly'' with the input flow rates. 
{We also provide a characterization 
of the rate of convergence to optimality 
under {\em GSS}, which is tighter than the conventional 
fluid-scale convergence rate.}


\subsection{Related Works}
{In this section, we discuss related works, and put our results in perspective.}

The most closely related work is \cite{St2012}, 
where the model considered in this paper was proposed, 
and a related problem was studied. 
{In both this paper and \cite{St2012}, 
the asymptotic regime of interest 
is when the input flow rates grow to infinity, 
and the system is considered under the {\em fluid scaling}, 
i.e., when the system states are scaled down by 
the input flow rates.}
In \cite{St2012}, the problem of interest 
is minimizing $\sum_{\bk} X_{\bk}^{1+\alpha}$, 
where $\alpha > 0$, and $X_{\bk}$ is the number of 
occupied servers with configuration $\bk$. 
A simple policy called {\em Greedy} was introduced, 
which asymptotically minimizes the sum $\sum_{\bk} X_{\bk}^{1+\alpha}$, 
for any $\alpha > 0$, in the stationary regime. 
{Policies {\em Greedy} and {\em GSS} differ in two important aspects. 
First, they try to minimize different objectives -- $\sum_{\bk} X_{\bk}^{1+\alpha}$ ($\alpha > 0$) and $\sum_{\bk} X_{\bk}$, respectively.} 
When $\alpha > 0$ is small, {\em Greedy} approximately solves 
the problem of minimizing the total number of occupied servers 
$\sum_{\bk} X_{\bk}$, in the asymptotic regime 
where the input flow rates grow to infinity, 
{and at the fluid scale}. 
However, 
if minimizing $\sum_{\bk} X_{\bk}$ is the ``true'' desired objective,
$\alpha > 0$ need to be chosen carefully, depending on the system scale (input flow
rates), which may be difficult to do. 
Therefore, we believe that asymptotically solving the exact problem of minimizing 
{$\sum_{\bk} X_{\bk}$} is 
of substantial interest. 
Moreover, the policy {\em GSS} proposed in this paper is as simple as {\em Greedy}, 
and uses no more system state information than {\em Greedy} does. 
{Second, at a technical level, to prove the asymptotic optimality of {\em Greedy}, 
\cite{St2012} considered only the fluid scaling and the corresponding fluid limits. 
In this paper, to prove the asymptotic optimality of {\em GSS}, 
it is no longer sufficient to consider the fluid-scale system behavior alone; 
a {\em local fluid scaling} is also considered, needed to study
the dynamics of safety stocks.
In addition, this allows us to derive a tighter characterization of the rate of convergence to optimality under {\em GSS},
as opposed to the fluid-scale convergence shown in \cite{St2012} for {\em Greedy}.} 

On a broader level, the model considered in this paper is related to the vast literature 
on classical stochastic bin packing problems. 
In a bin packing system, random-sized items arrive, 
and need to be placed into finite-sized bins. 
The items do not leave or move between bins, 
and a typical objective is to minimize the number of occupied bins. 
A packing problem is {\em one-dimensional} if sizes 
of the items and bins are captured by scalars, 
and {\em multi-dimensional} if they are captured by vectors. 
Problems with the 
multi-dimensional packing constraints  \eqn{eq:vec-packing} are called {\em
vector packing}. For a good review of one-dimensional bin packing, 
see for example \cite{Csirik2006}, and see for example \cite{Bansal2009} 
for a recent review of multi-dimensional packing. 
In bin packing {\em service} systems, items (customers) arrive at random times 
to be placed into a bin (server), and leave after a random service time. 
The servers can process multiple customers as long as packing constraints 
are observed. Customers get queued, and a typical objective of a packing algorithm is 
to maximize system throughput. 
(See for example \cite{Gamarnik2004} for a review of this 
line of work.) Our model is similar to the latter systems, except there are multiple bins (servers) 
-- in fact, an infinite number in our case. 
Models of this type are more recent (see for example, \cite{Jiang2012, Maguluri2012}). 
\cite{Jiang2012} addresses a joint routing and VM placement problem, 
which in particular includes packing constraints. The approach of \cite{Jiang2012} 
resembles Markov Chain algorithms used in combinatorial optimization.
\cite{Maguluri2012} considers maximizing throughput of a queueing system 
with a finite number of bins (servers), where VMs can wait for service. 
Very recently, \cite{GR2012} has new results on the classical one-dimensional
online bin packing; it also contains heuristics and simulations for the corresponding system with item departures, which is a special case of our model.

{As mentioned earlier, we consider t}he asymptotic regime where the input flow rates scale up to infinity.
In this respect, our work is related to the (also vast) literature on queueing systems in the 
{\em many servers} regime. (See e.g. \cite{ST2010_04} for an overview. The name ``many servers''
reflects the fact that the 
average number of occupied servers scales up to infinity 
as well, 
linearly with the input {flow} rates.) 
However, packing constraints 
are not present in earlier works (prior to \cite{St2012}) on the many servers regime, to the best of our knowledge. 

{The idea of maintaining sublinear safety stocks to increase system flexibility, and hence avoid  ``resource'' starvation -- the approach taken by {\em GSS}, the policy proposed in this paper -- has also appeared in other works. For example, 
see \cite{Meyn2005} and the references therein for an overview. 
However, to the best of our knowledge, the following feature of {\em GSS} is novel, 
and has not appeared in algorithms proposed in earlier works.
Namely, {\em GSS} creates safety stocks {\em automatically}, 
 in the sense 
that it does not require {\em a priori} knowledge of the subset of configurations
for which the sublinear safety stocks need to be maintained. As a result,
{\em GSS} does not require any {\em a priori} knowledge of the system parameters, 
because the safety stocks automatically adapt to parameter changes.
We remark that the policy {\em Greedy} proposed in \cite{St2012} 
also creates safety stocks, but they scale linearly with the input flow rates, whereas {\em GSS} creates sublinear safety stocks. }

{
Finally, an overview of some resource allocation issues 
that arise from VM placement in the context of cloud computing 
can be found in \cite{Gulati2012}.
}

\subsection{Organization}
The rest of the paper is organized as follows. 
In Section \ref{ssec:notation}, we introduce the notation and conventions 
adopted in this paper. The precise model and main results are 
described in Section \ref{sec:model-results}. The model is introduced 
in Section \ref{sec:model}. Here we describe 
two versions of the model, the closed and open system. 
In Section \ref{ssec:asymp-regime}, 
we describe the asymptotic regime of interest. 
The {\em GSS} policy is described 
in Section \ref{sec-GSS-definition}, and the main results, 
Theorems \ref{thm:main-closed} and \ref{th-main-res-open}, 
are stated in Section \ref{ssec:results}, for the closed and open system, 
respectively. Sections \ref{sec-closed} and \ref{sec-open} are devoted to 
proving Theorems \ref{thm:main-closed} and \ref{th-main-res-open}, respectively. 
A discussion of the results in this paper and some future directions 
is provided in Section \ref{sec-discussion}.

\subsection{Notation and Conventions}
\label{ssec:notation}

Let $\R$ be the set of real numbers, 
and let $\R_+$ be the set of nonnegative real numbers. 
Let $\Z$ be the set of integers, 
let $\Z_+$ be the set of nonnegative integers, 
and let $\N$ be the set of natural numbers. 
$\R^n$ denotes the real vector space of dimension $n$, and 
$\R_+^n$ denotes the nonnegative orthant of $\R^n$. 
$\Z^n$ and $\Z_+^n$ are similarly defined. 
We reserve bold letters for vectors, 
and plain letters for scalars and sets. 
For a scalar $x$, let $|x|$ denote its absolute value, 
and let $\lceil x\rceil$ denote the largest integer that does not exceed $x$.
For two scalars $x$ and $y$, let $x \wedge y = \min\{x, y\}$, 
and let $x \vee y = \max\{x, y\}$. 
For a vector $\bx = (x_i)_{i=1}^n \in \R^n$, 
let $\|\bx\|$ denote its $1$-norm, i.e., $\|\bx\| = \sum_{i=1}^n |x_i|$.
The distance from vector $\bx \in \R^n$ to a set $U \subset \R^n$ is denoted by 
$d(\bx,U)=\inf_{\bu\in U} \|\bx-\bu\|$. 
We use $\be_i$ to denote the $i$-th standard unit vector, 
with only the $i$th component being $1$, and all other components being $0$. 
For a set $\mathcal{N}$, let $\bOne_{\mathcal{N}}$ be the indicator function of $\mathcal{N}$. 
For a finite set $\mathcal{N}$, let $|\mathcal{N}|$ 
be its cardinality. 
For two sets $\mathcal{N}$ and $\mathcal{M}$, 
let $\mathcal{N} \backslash \mathcal{M}$ denote the 
set difference of $\mathcal{N}$ and $\mathcal{M}$, 
i.e., $\mathcal{N} \backslash \mathcal{M} = \{x\in \mathcal{N} : x \notin \mathcal{M}\}$. 
For a set $\mathcal{N} \subset \R^n$, 
let $\hull{\mathcal{N}}$ denote its convex hull, i.e., 
the set of all $\bx \in \R^n$ such that there exist 
$\gamma_1, \ldots, \gamma_m \in \R_+$ and  
$\bv_1, \ldots, \bv_m \in \mathcal{N}$ with 
$\bx = \sum_{j=1}^m \gamma_j \bv_j$ and 
$\sum_{j=1}^m \gamma_j = 1$.
Symbol $\to$ means ordinary convergence in $\R^n$, and $\implies$
denotes convergence in distribution of random variables taking values in $\R^n$,
equipped with the Borel $\sigma$-algebra. 
The abbreviation {\em w.p.1} means convergence {\em with probability 1}.
We often write $x(\cdot)$ to mean the function (or random process) $\{x(t),~t\ge 0\}$.
We write iff as a shorthand for ``if and only if'', i.o for ``infinitely often'', 
LHS for ``left-hand side'' and RHS for ``right-hand side''. 
{We also write WLOG for ``without loss of generality'', w.r.t for ``with respect to'', 
and u.o.c for ``uniformly on compact sets''.}

Throughout this paper, if $x(\cdot)$ is a random process (which in most cases
will be Markov), we will denote by $x(\infty)$ its random state when the process
is in stationary regime; in other words, $x(\infty)$ is equal in distribution
to $x(t)$ (for any $t$) when $x(\cdot)$ is stationary. We use the terms {\em steady state} and {\em stationary regime} interchangeably. 

\section{Model and Main Results}\label{sec:model-results}

\subsection{Infinite Server System with Packing Constraints}
\label{sec:model}
We consider the following infinite server system that evolves in continuous time. 
There are $I$ types of customers, indexed by $i \in \{1,2,\ldots,I\} \equiv \ci$, 
and an infinite number of homogeneous servers. 
A server can potentially serve more than one customer simultaneously. 
We use $\bk = (k_1, k_2, \ldots, k_I) \in \Z_+^I$, 
an $I$-dimensional vector with nonnegative integer components, 
to denote a \emph{server configuration}. 
The general packing constraints are captured by the finite set $\bar\ck \subset \Z_+^I$ 
of \emph{feasible server configurations}. 
Thus, a server can simultaneously serve $k_i$ customers of type $i$, $i \in \ci$, 
iff $\bk = (k_1, k_2, \ldots, k_I) \in \bar\ck$. From now on, we drop the word ``feasible'', 
and simply call $\bar\ck$ the set of server configurations. 

In this paper, we assume that the set $\bar\ck$ is \emph{monotone}. 
\begin{assumption}\label{asmp:monotone}
$\bar\ck$ is \emph{monotone} in the following sense. 
If $\bk \in \bar\ck$, and $\bk' \in \Z_+^I$ 
has $\bk' \leq \bk$ component-wise, 
then $\bk' \in \bar\ck$ as well.
\end{assumption}
A simple consequence of the monotonicity assumption is 
that $\bZero \in \bar\ck$. 
We now let $\ck = \bar\ck \backslash \{\bZero\}$ denote 
the set of non-zero server configurations. \\\\
{\bf Vector Packing is Monotone.} An important example 
of monotone packing is vector packing. 
Consider the vector packing constraints in \eqref{eq:vec-packing}. 
It is clear that if the server configuration 
$\bk = \{k_1, \ldots, k_I\}$ satisfies \eqref{eq:vec-packing}, 
and if $\bk' \le \bk$ component-wise, then $\bk'$ also satisfies \eqref{eq:vec-packing}. 
On the other hand, not all monotone packing is vector packing. 
For example, when $I = 2$, $\bar{\ck} = \{(0, 0), (0, 1), (0, 2), (1, 0), (2, 0)\}$ is monotone, but 
is not described by vector packing constraints. 
In the sequel, we will only assume monotone packing in our model, 
and all our results hold under this general setting.

To exclude triviality, we also assume that 
for all $i \in \ci$, $\be_i$ (the $i$-th standard unit vector) is an element of $\bar\ck$.

As discussed in the introduction, we make the following important assumption in this paper. 
We assume that simultaneous services do {\em not} affect the service distributions of individual customers; in other words, the service time of a customer is unaffected by 
whether or not there are other customers served simultaneously by the same server. 
Let us also remark that ideally, we would like to consider an open system, where 
each arriving customer is immediately placed for service in one of the servers, 
and leaves the system after service completion. 
However, we will first consider a ``closed'' version 
of this open system. The reason is twofold. 
First, the analysis of the closed system is a stepping stone 
to that of the open system, and illustrates the main ideas more clearly. 
Second, we will see shortly that the closed system can be used to model 
job migration in a cloud, and is therefore of independent interest. 

Denote by $X_{\bk}$ the number of servers with configuration $\bk \in \ck$. 
The system state is then the vector $\bX = \{X_{\bk}, ~\bk \in \ck\}$. 
By convention, $X_{\bZero} \equiv 0$ at all times.\\\\
{\bf Closed System.} 
Here we describe the ``closed'' version of the model. 
Let $r \in \N$ be given. Suppose that 
there are in total $r$ customers in the system, 
and no exogenous arrivals. 
For each $i \in \ci$, we suppose that 
there are $\rho_i r$ customers of type $i$ in the system 
at all times. This in particular implies that $\sum_{i \in \ci} \rho_i = 1$. 
It is convenient to index the system by $r$ its total number of customers, 
and we use $\bX^r = (X_{\bk}^r,~\bk \in \ck)$ to denote 
a system state. 
The system evolves as follows. Each customer is almost always in service, 
except at a discrete set of time instances, where it migrates from one server 
to another (possibly the same one){, subject to the packing constraints imposed by $\bar{\ck}$}.  
For a customer, the time between consecutive migrations is called its 
{\em service requirement}. 
Thus, one can alternatively think of a customer as departing the system 
after its service requirement, and then immediately arriving to the system, 
to be placed into a server. 
For each $i$, we assume that the service requirements of type-$i$ customers are i.i.d. 
exponential random variables with mean $1/\mu_i$, 
and that the service requirements are independent across different $i \in \ci$. 
A (Markovian) \emph{service policy} (``packing rule'') decides to which server a customer 
will be placed after its service requirement, based only on the current 
system state $\bX^r$. 
A service policy has to observe the packing constraints.
Under any well-defined service policy, the system state at time $t$, $\bX^r(t)$, 
is a continuous-time Markov chain on a finite state space.
{Hence, for each $r$,} the process $\{\bX^r(t), ~t\ge 0\}$ always has a stationary distribution.
\newline
\newline
{\bf Open System.} In the open system, customers of type $i$ arrive exogenously as an independent Poisson 
flow of rate $\lambda_i r$, where $\lambda_i$ is fixed and $r$ is a scaling parameter. Each arriving customer has to be placed for service immediately
in one of the servers, subject to the packing constraints imposed by $\bar{\ck}$. Service times of all customers are independent. Service time of
a type-$i$ customer is exponentially distributed with mean $1/\mu_i$. After a service completion, each customer leaves the system.
If we denote $\rho_i = \lambda_i/\mu_i$, then in steady state, the average number of type $i$ customers in the system is $\rho_i r$, and the average total number
of customers is $\sum_ i \rho_i r$. We assume, WLOG, that $\sum_ i \rho_i = 1$ -- this is equivalent to re-choosing the value of parameter $r$, 
if necessary. A (Markovian) \emph{service policy} (``packing rule'') in this case decides to which server an arriving customer 
will be placed, based only on the current system state. A service policy has to observe the packing constraints. 
{Similar to the closed system, 
we let $X_{\bk}^r(t)$ denote the number of servers with configuration $\bk$ at time $t$ 
in the $r$th system. However, for the policy that we will study,
$\bX^r(t) = (X_{\bk}^r(t))_{\bk \in \ck}$ 
will not be a Markov process.}
We postpone the discussion of {a complete Markovian description 
of the system} and the existence 
of {the associated} stationary distribution to {Section \ref{sssec:gss-open}}.

\subsection{Asymptotic Regime}\label{ssec:asymp-regime}
We are interested in finding a service policy that minimizes  
the total number of occupied servers in the stationary regime. 
The exact problem is intractable, so instead we consider asymptotically 
optimal service policies. 
For both the closed and open systems, the asymptotic regime of interest 
is when $r \rightarrow \infty$. Informally speaking, in this limit, 
the {\em fluid-scaled} system state satisfies a conservation law (cf. Eq. \eqref{eq:conservation}), 
and the best that a policy can do is solving a linear program, subject to this conservation law. 
We now describe the asymptotic regime in more detail. 

First, we defined the so-called {\em fluid scaling}. 
Recall that both the closed and open systems are indexed by $r$, 
and $\bX^r(t)$ is the vector that denotes the numbers of servers at time $t$, 
in the $r$th system. 
The {\em fluid scaled} process is $\bx^r(t)=\bX^r(t)/r$. 
For each $r$, in the closed system, 
$\bX^r(\cdot)$ has a (not necessarily unique) stationary distribution, 
so $\bx^r(\cdot)$ also has a stationary distribution. 
We will see shortly that in an open system, 
$\bX^r(\cdot)$ also has a stationary distribution 
(see Lemma \ref{lem-complete-state-tight}). 
Denote by $\bX^r(\infty)$ 
and $\bx^r(\infty)$ the random states of the corresponding processes 
in a stationary regime. (Recall the convention in Section~\ref{ssec:notation}.)

We now argue that as $r \to \infty$, 
\begin{equation}\label{eq:approx-conservation}
\sum_{\bk \in \ck} k_i x^r_{\bk}(\infty) \implies \rho_i, \mbox{ for all } i. 
\end{equation}

In a closed system, for each $i \in \ci$, there are $\rho_i r$ customers 
of type $i$ in the system at all times, so on all sample paths, 
\[
\sum_{\bk \in \ck} k_i x^r_{\bk}(t) = \rho_i, \mbox{ for all } r, t \mbox{ and } i. 
\]
This implies that the same holds for $\bx^r(\infty)$. In an open system, 
the total number of type-$i$ customers is $\sum_{\bk \in \ck} k_i X^r_{\bk}(\infty)$, 
in steady state.  
It is easy to see that, independent from the service policy, 
this quantity is a Poisson random variable with mean $\rho_i r$. 
Thus, as $r \rightarrow \infty$, $\sum_{\bk \in \ck} k_i x^r_{\bk}(\infty) \implies \rho_i$.

Now consider the following linear program (LP). 
\begin{eqnarray}
\mbox{Minimize} & & \sum_{\bk\in \ck} x_{\bk} \\
\mbox{subject to} & & \sum_{\bk \in \ck} k_i x_{\bk} = \rho_i, \quad \mbox{for all } i \in \ci, \label{eq:conservation} \\
			    & & \quad \quad \ \ x_{\bk} \geq 0, \quad \mbox{ for all } \bk \in \ck.
\end{eqnarray}
Denote by $\cx$ the set of feasible solutions to LP: 
$$\cx = \{\bx \in \R_+^{|\ck|} : \sum_{\bk \in \ck} k_ix_{\bk} = \rho_i, i \in \ci\}.$$
Then $\cx$ is a compact subset of $\R_+^{|\ck|}$. 
Let $\cx^*$ denote the set of optimal solutions of LP, 
and let $u^*$ denote its optimal value. 
In light of Eqs. \eqref{eq:approx-conservation} 
and \eqref{eq:conservation}, a service policy is asymptotically optimal 
if, roughly speaking, under this policy and for large $r$, 
$\sum_{\bk \in \ck} x^r_{\bk}(\infty) \approx u^*$ with high probability 
(cf. Theorems \ref{thm:main-closed} and \ref{th-main-res-open}). 

The following characterization of the set $\cx^*$ by dual variables 
will be useful. The proof is elementary and omitted.
\begin{lem}\label{lem:lp-dual-char}
$\bx = (x_{\bk})_{\bk \in \ck} \in \cx^*$ iff 
$\bx$ is a feasible solution of LP, and 
there exist $\eta_i \in \R$, $i \in \ci$, such that 
\begin{itemize}
\item[(i)] $\sum_{i \in \ci} k_i \eta_i \leq 1$ for all $\bk \in \ck$, and 
\item[(ii)] if $\sum_{i \in \ci} k_i \eta_i < 1$, then $x_{\bk} = 0$. 
\end{itemize}
\end{lem}
The following lemma relates the distance between a point $\bx \in \cx$ 
and the optimal set $\cx^*$ to the objective value of LP evaluated at $\bx$.
\begin{lem}\label{LEM:LP-RATE}
There exists a positive constant $D \ge 1$ such that for any $\bx \in \cx$, 
\[
D\left(\sum_{\bk \in \ck} x_{\bk} - u^*\right) \ge d\left(\bx, \cx^*\right).
\]
\end{lem}
Note that $D \ge 1$ is necessary, 
since for every $\bx \in \cx$, 
$d\left(\bx, \cx^*\right) \ge \sum_{\bk \in \ck} x_{\bk} - u^*$.
\begin{proof} {See Appendix \ref{apdx:lp-rate}.}
\end{proof}


\subsection{Greedy with {sublinear} Safety Stocks ({GSS})}
\label{sec-GSS-definition}
Now we introduce the service policy, {\em Greedy with sublinear Safety Stocks (GSS)}, along with 
a variant, which we will prove to be asymptotically optimal.

\subsubsection{{GSS} Policy in a Closed System}\label{sssec:gss-closed}

\noindent {\bf {\em GSS}.} Let $p \in (\frac{1}{2}, 1)$. For a given $r$, define a weight function 
$w^r : \R_+ \rightarrow \R_+$ to be 
$w^r(X) = 1 \wedge \frac{X}{r^p}$. 
Let $\cm$ denote the set of all pairs $(\bk, i) \in \ck \times \ci$ 
such that $\bk \in \ck$ and $\bk - \be_i \in \bar\ck$. 
Given $\bX = \{X_{\bk'}, \bk' \in \ck\}$ and $(\bk, i) \in \cm$, 
define $\Delta^r_{(\bk, i)}(\bX) = w^r\left(X_{\bk}\right) - w^r(X_{\bk - \be_i})$.
Under {\em GSS}, a customer of type $i$ is placed into a server with configuration $\bk - \be_i$ 
where $X_{\bk - \be_i} > 0$ or $\bk - \be_i = \bZero$, such that $\Delta_{(\bk, i)}(\bX)$ is minimal. 
Ties are broken arbitrarily. 

Note that the {\em GSS} policy makes decisions based only the current system state.
The parameter $r$ which it uses is nothing else but the total number of customers
in the system, which is, of course, a function of the state, and which happens to be 
constant in the closed system.

We now provide an intuitive explanation of the policy. 
Let $f^r$ be the anti-derivative of $w^r$, so that
\[
f^r(X) = \left\{\begin{array}{ll}
\frac{X^2}{2 r^p}, & \mbox{ if } X \in [0, r^p]; \\
X - \frac{r^p}{2}, & \mbox{ if } X > r^p.
\end{array}\right.
\]
Let $F^r(\bX) = \sum_{\bk \in \ck} f^r(X_{\bk})$. 
Then $w^r$ and $\Delta^r_{(\bk, i)}$ capture the first-order change in $F^r$. 
Suppose that the current system state is $\bX = (X_{\bk})_{\bk \in \ck}$. 
Then, placing a type-$i$ customer into a server with configuration $\bk - \be_i$ 
only changes $X_{\bk-\be_i}$ and $X_{\bk}$: 
$X_{\bk - \be_i}$ decreases by $1$ (if $X_{\bk - \be_i} > 0$), 
and $X_{\bk}$ increases by $1$. Thus, the first-order change in $F^r$ is 
\[
\frac{d}{dX} f^r(X) \Big|_{X = X_{\bk}} - \frac{d}{dX} f^r(X) \Big|_{X = X_{\bk - \be_i}} 
= \Delta^r_{(\bk, i)}(\bX).
\]
In this sense, {\em GSS} decreases $F^r$ greedily, by placing a customer 
into a server that results in the largest (first-order) decrease in $F^r$. 

The next lemma states that $F^r(\bX)$ only differs from $\sum_{\bk} X_{\bk}$ 
by $O(r^p)$. The proof is straightforward and omitted.

\begin{lem}\label{lem:F-sum-close}
For any $\bX \in \R_+^{|\ck|}$, 
\[
\sum_{\bk\in \ck} X_{\bk} - \frac{|\ck| r^p}{2} \leq F^r(\bX) \leq \sum_{\bk \in \ck} X_{\bk}.
\]
\end{lem}

Under the fluid scaling described earlier,
 the difference $O(r^p)$ between $F^r(\bX)$ and $\sum_{\bk \in \ck} X_{\bk}$ 
becomes negligible, as it is of order $o(r)$. 
Thus, for a fluid-scaled process, minimizing
$F^r(\bX)$ (what {\em GSS} tries to do) is ``equivalent'' to
minimizing $\sum_{\bk \in \ck} X_{\bk}$, when $r$ is large.

\subsubsection{{GSS} Policy in an Open System}\label{sssec:gss-open}

First, we describe the ``pure'' {\em GSS} policy.
\newline
\newline
{\bf {\em GSS}.} 
Let $p \in (\frac{1}{2}, 1)$.
For a given system state $\bX$, let $Z=Z(\bX)$ denote the total number of customers in the system. For a system with parameter $r$,
define a weight function $\bar w^r(X)=\bar w^r(X;Z)$ as follows: $\bar w^r(X) = 1 \wedge \frac{X}{Z^p}$. 
(Note that $\bar w^r(X)$ generalizes the corresponding weight function $w^r(X) = 1 \wedge \frac{X}{r^p}$
for the closed system, because in the closed system with parameter $r$
the total number of customers
is constant $Z\equiv r$.)
Let $\cm$ denote the set of all pairs $(\bk, i) \in \ck \times \ci$ 
such that $\bk \in \ck$ and $\bk - \be_i \in \bar\ck$. 
Given $\bX = \{X_{\bk'}, \bk' \in \ck\}$ and $(\bk, i) \in \cm$, 
define $\bar \Delta^r_{(\bk, i)}(\bX) = \bar w^r\left(X_{\bk}\right) - \bar w^r(X_{\bk - \be_i})$.
Under {\em GSS}, an arriving customer of type $i$ is placed into a server with configuration $\bk - \be_i$ 
where $X_{\bk - \be_i} > 0$ or $\bk - \be_i = \bZero$, such that $\bar \Delta_{(\bk, i)}(\bX)$ is minimal. 
Ties are broken arbitrarily.

In this paper, for the open system, we will analyze not the ``pure'' {\em GSS} policy,
described above,
but its slight modification, called {\em Modified GSS} ({\em GSS-M}).
\newline\newline
{\bf {\em GSS-M}.} Under this policy, a {\em token} of type $i$ is generated immediately upon each service completion of type $i$, and is placed for ``service'' immediately
according to {\em GSS}. The system state $\bX = \{X_{\bk}, \bk\in \ck\}$ account for both tokens 
of type $i$ as well as actual type-$i$ customers for all $i \in \ci$.
Each arriving type $i$ customer first seeks to replace an existing token
of type $i$ already in ``service''  (chosen arbitrarily),
and if there is none, it is placed for service according to {\em GSS}. Each token that is not replaced by an actual arriving customer before
an independent exponentially distributed timeout with mean $1/\mu_0$, leaves the system.
(This modification is the same as the one introduced  in \cite{St2012} for the {\em Greedy} algorithm, to obtain the {\em Greedy-M} policy.) 

We emphasize that {\em GSS} and {\em GSS-M} do {\em not} require the knowledge of parameter $r$.

Since the system evolution under the {\em GSS-M} involves both actual customers and tokens,
we need to define the Markov chain describing this evolution more precisely.
A {\em complete server configuration} is defined (in the same way as in \cite{St2012}) as
a pair $(\bk,\hat \bk)$, where vector $\bk=(k_1,\ldots,k_I)\in \ck$ gives the numbers of all customers
(both actual and tokens) in a server, while vector $\hat \bk \le \bk$, $\bk\in \bar \ck$,
gives the numbers
of actual customers only. 
The Markov process state at time $t$ is
the vector $\{X_{(\bk,\hat \bk)}^r(t)\}$, where the index $(\bk,\hat \bk)$ takes values
that are all possible complete server configurations, and 
superscript $r$, as usual, indicates the system with parameter $r$.
Note that $\bX^r(t) = \{X^r_{\bk}(t), \bk \in \ck\}$ can be considered as a ``projection'' of 
$\{X^r_{(\bk,\hat \bk)}(t)\}$, with $X^r_{\bk}(t) = \sum_{\hat \bk: \hat \bk \le \bk} X^r_{(\bk, \hat \bk)}$ 
for each $\bk \in \ck$.
{Let $\hat Y_i^r(t)$, $\tilde Y_i^r(t)$, and $Y_i^r(t)=\hat Y_i^r(t)+\tilde Y_i^r(t)$ 
denote the total number of actual type-$i$ customers, the
total number of type-$i$ tokens, and
the total number of all (both actual and tokens) 
type-$i$ customers in the $r$th system, 
respectively.} 
The total number of actual customers of all types 
is then $Z^r(t)=\sum_i \hat Y^r_i(t)$.
The behaviors of the processes $\{(Y_i^r(t), \hat Y_i^r(t)), ~t\ge 0\}$, are independent
across all $i$, with $\hat Y_i^r(\infty)$ 
having Poisson distribution with mean $\rho_i r$.
The following fact has the same proof as Lemma 11 in \cite{St2012}.
\begin{lem}
\label{lem-complete-state-tight}
{The} Markov chain 
$\{X_{(\bk,\hat \bk)}^r(t)\}, ~t\ge 0$,
is irreducible and positive recurrent for each $r$. 
\end{lem}

\noindent {\bf Remark.} Informally, the reason
(which is the same as in \cite{St2012})
for considering a modified version of {\em GSS} instead of pure {\em GSS} in an open system
is as follows. 
Recall that in a closed system, a customer migration can be also thought of as 
its departure followed immediately by an arrival of the same type. As such, 
departures and arrivals in a closed system are perfectly ``synchronized'', 
which in particular means that in a closed system, for every departing customer, 
we always have the option of putting it right back into the server which it has just 
departed from. This means that a greedy control, pursuing minimization 
of a given objective function, cannot  possibly increase (up to a first-order
approximation) the objective function at every customer migration.
In contrast, in an open system, departures and arrivals are not synchronized. 
Therefore, it is not immediately clear that a greedy algorithm
will necessarily improve the objective.
The tokens are introduced so that, informally speaking,
the decisions on placements of new type-$i$ arrivals are made somewhat ``in advance'',
at the times of prior type-$i$ departures. In this sense, the behavior of an open system
``emulates'' that of a corresponding closed system.

\subsection{Main Results}\label{ssec:results}
\begin{thm}\label{thm:main-closed}
Let $p \in (\frac{1}{2}, 1)$. 
For each $r$, consider the closed system operating under {\em GSS} policy, in steady state.
Then there exists some constant $C > 0$, 
not depending on $r$, such that 
\[
\pr\left(d(\bx^r(\infty), \cx^*) \le C r^{p-1}\right) \rightarrow 1
\]
as $r \rightarrow \infty$. Consequently, we have fluid-scale asymptotic optimality:
$$
d(\bx^r(\infty), \cx^*) \implies 0.
$$
\end{thm}

\begin{thm}
\label{th-main-res-open}
Let $p \in (\frac{1}{2}, 1)$. 
For each $r$, consider the open system operating under {\em GSS-M} policy, in steady state.
{Then there exists some constant $C > 0$, 
not depending on $r$, such that} as $r\to\infty$,
\beql{eq-main-res-open}
\pr\left(d(\bx^r(\infty),\cx^*)\le C r^{p-1}\right) \to 1,
\end{equation}
and
\beql{eq-main-res-open2}
r^{-p} \sum_i \tilde Y_i^r(\infty) \implies  0.
\end{equation}
Consequently, we have fluid-scale asymptotic optimality:
$$
d(\bx^r(\infty),\cx^*) \implies 0 {~~~~\mbox{and}~~~~
r^{-1} \sum_i \tilde Y_i^r(\infty) \implies  0}.
$$
\end{thm}

\section{Closed System: Asymptotic \\Optimality of {GSS}}\label{sec-closed}
{We restrict our attention to closed systems and prove Theorem \ref{thm:main-closed} in this section. 
As mentioned earlier, it is not sufficient to consider only the system states 
at the fluid scale, defined in Section \ref{ssec:asymp-regime}. 
We also need the concept of \emph{local fluid scaling}, introduced below. 
Proposition \ref{prop:opt-si-1} -- a key step in the proof of Theorem \ref{thm:main-closed} 
-- is established in Section \ref{ssec:key-prop}. 
In Section \ref{ssec:proof-closed}, we construct an appropriate probability space, 
quantify the drift of $F^r$ under {\em GSS} 
(cf. Propositions \ref{prop:loc-decr} and \ref{PROP:CONV-PROB}), 
and prove Theorem \ref{thm:main-closed}.}

\subsection{Local Fluid Scaling}\label{ssec:loc-fluid-scaling}
Besides the fluid-scaled processes $\bx^r(t)$ defined in Section \ref{ssec:asymp-regime}, 
it is also convenient to consider the system dynamics 
at the {\em local fluid scale}. 
More precisely, for each $r$ and $t$, 
define the corresponding {\em local fluid scale} process $\wt{\bx}^r(t)$ by 
\[
\wt{\bx}^r(t) = \frac{1}{r^p} \bX^r (t).
\]
In the asymptotic regime $r \rightarrow \infty$, 
recall that the fluid scale process $\bx^r(\cdot)$ 
always lives in the compact set $\cx$ (defined in Section \ref{ssec:asymp-regime}). 
This is no longer true for the local fluid scale processes $\wt{\bx}^r(\cdot)$: 
for a fixed $t$, $\{\wt{\bx}^r(t)\}_r$ can be unbounded. 
However, at the local fluid scale, we will always consider 
the following weight function $\wt{w}$, 
which remains bounded. 

Define the local-fluid-scale weight function $\wt{w} : \R \cup \{\infty\} \rightarrow \R_+$ 
to be $\wt{w}(\wt{x}) = 1\wedge \wt{x}$. By convention, $1 < \infty$, 
so $\wt{w}$ is well-defined. 
Note that for every $r$, $\wt{w}(\wt{x}^r) = w^r(X^r)$, 
where $\wt{x}^r = X^r/r^p$. 
For $(\bk, i) \in \cm$, we can also define the weight difference at the local fluid scale to be
\[
\Delta_{(\bk, i)}(\wt{\bx}) = \wt{w}(\wt{x}_{\bk}) - \wt{w}(\wt{x}_{\bk - \be_i}).
\]

\noindent {\bf Remark}. In the sequel, we will always use lower case $x$ (or $\bx$) to denote 
quantities at the fluid scale, $\wt{x}$ (or $\wt{\bx}$) to denote quantities at the local fluid scale, 
and upper case $X$ (or $\bX$) to denote quantities without scaling.

\subsection{Key Proposition}\label{ssec:key-prop}
For a vector $\wt{\bx} \in \left(\R_+ \cup \{\infty\} \right)^{|\ck|}$ 
with components being possibly infinite, 
we can define the concept of a {\em Strictly Improving (SI) pair associated with $\wt{\bx}$}. 

\begin{definition}[Strictly Improving (SI) pair]\label{df:si-pair}
For \\$(\bk, i)$, $(\bk', i) \in \cm$, $\{(\bk, i), (\bk', i)\}$ is an SI pair associated with $\wt{\bx}$ if
\begin{itemize}
\item[(a)] $k_i \geq 1$, $\wt{x}_{\bk} > 0$;
\item[(b)] either $\bk' = \be_i$, or $[k'_i > 0 \mbox{ and } \wt{x}_{\bk' - \be_i} > 0]$; and 
\item[(c)] $\Delta_{(\bk', i)} < \Delta_{(\bk, i)}$.
\end{itemize}
\end{definition}

{
The idea of SI pairs is as follows. 
Suppose that the current system state is $\bX^r$, and a type-$i$ customer 
just completed its service requirement at a server with configuration $\bk$. 
Then the first-order change in $F^r$ is $-\Delta^r_{(\bk, i)}(\bX^r)$.
Suppose that this customer is then placed into a server with configuration $\bk'$, 
under {\em GSS}. 
Then, the total (first-order) change in $F^r$ after this transition is 
$\Delta^r_{(\bk', i)}(\bX^r) - \Delta^r_{(\bk, i)}(\bX^r)$, or 
$\Delta_{(\bk', i)}(\wt{\bx}^r) - \Delta_{(\bk, i)}(\wt{\bx}^r)$. 
The existence of an SI pair ensures that we can always improve 
(up to first order)
the current value of $F^r$.}


Recall that for any feasible system state $\bX^r$, $\bx^r = \bX^r/r$ denotes the 
fluid-scale system state, and $\wt{\bx}^r = \bX^r/r^p$ denotes the associated state 
at the local fluid scale. 
The following proposition establishes that whenever $\bx^r$ is sufficiently far away from 
optimality, an SI pair exists.

\begin{prop}\label{prop:opt-si-1}
Let $D > 0$ be the same as in Lemma \ref{LEM:LP-RATE}.
Then, there exist a positive constant $\varepsilon$ such that the following holds. 
For sufficiently large $r$, if 
$d(\bx^r, \cx^*) \ge 2D|\ck|r^{p-1}$, then there exists an SI pair $\{(\bk', i), (\bk, i)\}$ 
(possibly depending on $r$) associated with $\wt{\bx}^r = (\wt{x}_{\bk}^r)_{\bk \in \ck}$, 
and furthermore, $\wt{x}_{\bk}^r \geq \varepsilon$, $\wt{x}_{\bk' - \be_i}^r \ge \veps$, 
and $\Delta_{(\bk', i)}(\wt{\bx}^r) - \Delta_{(\bk, i)}(\wt{\bx}^r) \leq -\varepsilon$.
\end{prop}

Proposition \ref{prop:opt-si-1} follows from the two lemmas below. 

\begin{lem}\label{lem:existence-si-1}
Consider any sequence $\{\bx^r\}$ and the associated states $\wt{\bx}^r$. 
Let $\bx \in \cx$ be a limit point of the sequence $\{\bx^r\}$, 
so that the the subsequence $\{r_n\}$ of $\{r\}$ satisfies 
$\bx^{r_n} \rightarrow \bx$ and 
$\wt{\bx}^{r_n} \rightarrow \wt{\bx}$ as $n \rightarrow \infty$, 
with some components of $\wt{\bx}$ being possibly infinite. 
If there is no SI pair associated with $\wt{\bx}$, 
then $\bx \in \cx^*$, i.e. $\bx$ is an optimal solution of LP.
\end{lem}
\begin{proof}[of Lemma \ref{lem:existence-si-1}]
Suppose that there is no SI pair associated with $\wt{\bx}$. 
We will show that $\bx \in \cx^*$, i.e., $\bx$ is an optimal solution of the linear program LP. 
To this end, we will use Lemma \ref{lem:lp-dual-char}. 
In particular, we will construct $\eta_i \geq 0$, $i \in \ci$ such that 
\begin{itemize}
\item[(i)] $\sum_{i \in \ci} k_i \eta_i \leq 1$ for all $\bk \in \ck$, and 
\item[(ii)] if $\sum_{i \in \ci} k_i \eta_i < 1$, then $\wt{x}_{\bk} < 1$. 
\end{itemize}
Note that condition (ii) here is stronger than condition (ii) in Lemma \ref{lem:lp-dual-char}.

Let $\eta_i = \wt{w}(\wt{x}_{\be_i})$ for all $i \in \ci$. Then clearly $\eta_i \in [0, 1]$ for all $i \in \ci$. 
We first show that condition (i) holds. To this end, we prove the following stronger statement: 
if $\bk \in \ck$ is such that $k_i \ge 1$ implies $\eta_i > 0$, 
then $\sum_{i \in \ci} k_i \eta_i = \wt{w}(\wt{x}_{\bk})$. 
Suppose not. Let $\bk \in \ck$ be a minimal counterexample, 
so that 
\begin{equation}\label{eq:si-exist-1}
\sum_{i \in \ci} k_i \eta_i \neq \wt{w}(\wt{x}_{\bk}), 
\end{equation}
and for each $i \in \ci$, $k_i \ge 1$ implies $\eta_i > 0$.  
Note that $\sum_{i \in \ci} k_i \ge 2$, since $\eta_i = \wt{w}(\wt{x}_{\be_i})$ 
for each $i \in \ci$, by definition. Thus, there exists $i \in \ci$ 
such that $\eta_i > 0$, $\bk' = \bk - \be_i \in \ck$, and 
\begin{equation}\label{eq:si-exist-2}
\sum_{i \in \ci} k'_i \eta_i = \wt{w}(\wt{x}_{\bk'}).
\end{equation}

Subtracting Eq. \eqref{eq:si-exist-2} from Eq. \eqref{eq:si-exist-1}, 
we get that
\[
\Delta_{(\bk, i)} = \wt{w}(\wt{x}_{\bk}) - \wt{w}(\wt{x}_{\bk'}) \neq \eta_i. 
\]
Thus either $\Delta_{(\bk, i)} > \eta_i$, or $\Delta_{(\bk, i)} < \eta_i$. 
If $\Delta_{(\bk, i)} > \eta_i$, we verify that 
$\{(\bk, i), (\be_i, i)\}$ is an SI pair associated with $\wt{\bx}$. 
First, conditions (b) and (c) in Definition \ref{df:si-pair} are automatically satisfied. 
Second, $\Delta_{(\bk, i)} > \eta_i > 0$. 
In particular, $\wt{x}_{\bk} > 0$. We also have $k_i \geq 1$, so 
condition (a) in Definition \ref{df:si-pair} is also satisfied. 

If $\Delta_{(\bk, i)} < \eta_i$, we verify that 
$\{(\be_i, i), (\bk, i)\}$ is an SI pair associated with $\wt{\bx}$. 
First, condition (c) in Definition \ref{df:si-pair} is automatically satisfied. 
Second, since $\eta_i > 0$, $\wt{x}_{\be_i} > 0$. 
Thus condition (a) in Definition \ref{df:si-pair} is satisfied. 
Finally, $k_i \ge 1$ by assumption, so to verify condition (b), 
we only need to verify that $\wt{x}_{\bk - \be_i} > 0$. 
Since $\sum_{i \in \ci} k_i \geq 2$, 
$\sum_{i \in \ci} k'_i \geq 1$. 
This implies that there exists $i' \in \ci$ such that $k'_{i'} \ge 1$. 
Thus $k_{i'} \ge k'_{i'} \ge 1$, so $\eta_{i'} > 0$. 
By Eq. \eqref{eq:si-exist-2}, 
$\wt{w}(\wt{x}_{\bk'}) \ge \eta_{i'} > 0$, 
so $\wt{x}_{\bk'} > 0$.
Thus, condition (b) in Definition \ref{df:si-pair} 
is verified. 

In either case, we have an SI pair associated with $\wt{\bx}$, 
contradicting the assumption that there is no SI pair associated with $\wt{\bx}$. 
Thus, for all $\bk \in \ck$ such that $k_i \ge 1$ implies $\eta_i > 0$, 
\[
\sum_{i \in \ci} k_i \eta_i = \wt{w}(\wt{x}_{\bk}).
\]
For all $\bk \in \ck$, we can find $\bk' \le \bk$ such that 
$\bk' \in \ck$, $k'_i \ge 1$ implies $\eta_i > 0$, 
and $\sum_{i \in \ci} k_i \eta_i = \sum_{i \in \ci} k'_i \eta_i$. Thus, 
\[
\sum_{i \in \ci} k_i \eta_i = \sum_{i \in \ci} k'_i \eta_i = \wt{w}(\wt{x}_{\bk'}) \leq 1.
\]
This establishes condition (i). 

We now establish condition (ii). Suppose that condition (ii) does not hold. 
Let $\bk \in \ck$ be minimal such that 
\[
\wt{x}_{\bk} \geq 1,  ~~~\mbox{ and }~~~ \sum_{i \in \ci} k_i \eta_i < 1. 
\]
First, note that $\bk \neq \be_i$ for any $i \in \ci$, 
because if $\eta_i < 1$, then
\[
1 > \eta_i = \wt{w}(\wt{x}_{\be_i}) = 1 \wedge \wt{x}_{\be_i}.
\]
Thus $\sum_{i \in \ci} k_i \geq 2$.
Second, if $\eta_i > 0$ for all $i \in \ci$ with $k_i \geq 1$, 
then from the proof of condition (i), we have that 
\[
1 > \sum_{i \in \ci} k_i \eta_i = \wt{w}(\wt{x}_{\bk}) = 1 \wedge \wt{x}_{\bk},
\]
so we have $\wt{x}_{\bk} < 1$, reaching a contradiction. 
Thus, there exists $i \in \ci$ such that $\eta_i = 0$ and $k_i \geq 1$. 
Let $\bk' = \bk - \be_i$. Then $\bk' \in \ck$, since 
\[
\sum_{i \in \ci} k'_i = \sum_{i \in \ci} k_i - 1 \geq 1.
\]
Since $\eta_i = 0$, 
\[
\sum_{i \in \ci} k'_i \eta_i = \sum_{i \in \ci} k_i \eta_i < 1.
\] 
By minimality of $\bk$, we must have $\wt{x}_{\bk'} < 1$. 
Thus, $\wt{w}(\wt{x}_{\bk'}) = 1 \wedge \wt{x}_{\bk'} < 1$, 
and $\wt{w}(\wt{x}_{\bk}) = 1 \wedge \wt{x}_{\bk} = 1$. 
This implies that 
\[
\Delta_{(\bk, i)} > 0 = \eta_i,
\]
and that $\{(\bk, i), (\be_i, i)\}$ is an SI pair associated with $\wt{\bx}$. 
This is a contradiction, so condition (ii) is established. 
\end{proof}

\begin{lem}\label{lem:existence-si-2}
Consider any sequence $\{\bx^r\}$ and associated states $\wt{\bx}^r$. 
Let $\bx^{r_n}$, $\bx$, $\wt{\bx}^{r_n}$ and $\wt{\bx}$ 
be the same as in Lemma \ref{lem:existence-si-1}. 
If for all sufficiently large $n$, 
$d(\bx^{r_n}, \cx^*) \ge 2D|\ck|r_n^{p-1}$, 
then there is an SI pair associated with $\wt{\bx}$.
\end{lem}
\begin{proof}[of Lemma \ref{lem:existence-si-2}]
We prove the lemma by contradiction. 
Suppose that the lemma is not true, 
then for sufficiently large $n$, $d(\bx^{r_n}, \cx^*) \ge 2D |\ck| r_n^{p-1}$, 
and there is no SI pair associated with $\wt{\bx}$. 
By Lemma \ref{lem:existence-si-1}, $\bx$ 
is an optimal solution of LP, and 
from the proof of Lemma \ref{lem:existence-si-1}, $\boldeta = (\eta_i)_{i \in \ci}$ 
is an optimal dual solution of LP, where $\eta_i = \wt{x}_{\be_i}$ 
for all $i \in \ci$. 

For a given $r$, consider the following linear program, 
which we call $\mbox{LP}^r$. 
\begin{eqnarray}
\mbox{Minimize} & & \sum_{\bk\in \ck} \wt{x}_{\bk} \\
\mbox{subject to} & & \sum_{\bk \in \ck} k_i \wt{x}_{\bk} = \rho_i r^{1-p}, \quad \mbox{for all } i \in \ci, \\
			    & & \quad \quad \ \ \wt{x}_{\bk} \geq 0, \quad \quad \quad \mbox{ for all } \bk \in \ck.
\end{eqnarray}
$\mbox{LP}^r$ is just a scaled version of LP, defined in Section \ref{ssec:asymp-regime}. 
For each $r$, the feasible set of $\mbox{LP}^r$ is $r^{1-p}\cx$, 
its set of optimal solutions is $r^{1-p}\cx^*$, and 
its optimal value is $r^{1-p} u^*$. 
$r^{1-p}\bx$ is an optimal solution of $\mbox{LP}^r$, 
and $\boldeta$ is an optimal dual solution. 
Furthermore, by Lemma \ref{LEM:LP-RATE}, 
for sufficiently large $n$, 
\begin{eqnarray*}
\sum_{\bk \in \ck} \wt{x}^{r_n}_{\bk} - r^{1-p}u^* 
& = & r^{1-p} \left(\sum_{\bk \in \ck} x^{r_n}_{\bk} - u^*\right) \\
& \ge & r^{1-p} d(\bx^{r_n}, \cx^*)/D \\
& \ge & r^{1-p}\cdot (2D|\ck|r^{p-1})/D \ge 2|\ck|. 
\end{eqnarray*}
For each $n$, consider the Lagrangian $L(\wt{\bx}^{r_n}, \boldeta)$ of $\mbox{LP}^{r_n}$, 
evaluated at $\wt{\bx}^{r_n}$ and $\boldeta$: 
\[
L(\wt{\bx}^{r_n}, \boldeta) = \sum_{\bk \in \ck} \wt{x}^{r_n}_{\bk} + \sum_{i \in \ci} \eta_i\left(\rho_i r_n^{1-p} - \sum_{\bk \in \ck} k_i \wt{x}^{r_n}_{\bk}\right). 
\]
We calculate the Lagrangian in two ways. First, by feasibility of $\wt{\bx}^{r_n}$, 
$L(\wt{\bx}^{r_n}, \boldeta) = \sum_{\bk \in \ck} \wt{x}^{r_n}_{\bk}$. 
Second, we rewrite $L(\wt{\bx}^{r_n}, \boldeta)$ as 
\[
L(\wt{\bx}^{r_n}, \boldeta) =  r_n^{1-p} \sum_{i \in \ci} \rho_i \eta_i 
+ \sum_{\bk \in \ck} \left(1 - \sum_{i \in \ci} k_i \eta_i\right)\wt{x}^{r_n}_{\bk}. 
\]
The first term on the RHS equals $r_n^{1-p} u^*$, 
by the dual optimality of $\boldeta$. 
For the second term on the RHS, note that 
in the proof of Lemma \ref{lem:existence-si-1}, 
we have established that 
for all $\bk \in \ck$, 
$\sum_{i \in \ci} k_i \eta_i \le 1$, 
and if $\sum_{i\in \ci} k_i \eta_i < 1$, then $\wt{x}_{\bk} < 1$.
Since $\wt{\bx}^{r_n} \rightarrow \wt{\bx}$, 
for all sufficiently large $n$, 
if $\sum_{i\in \ci} k_i \eta_i < 1$, then $\wt{x}^{r_n}_{\bk} \le 1$.
Thus for all sufficiently large $n$, 
\[
\sum_{\bk \in \ck} \left(1 - \sum_{i \in \ci} k_i \eta_i\right)\wt{x}^{r_n}_{\bk} \leq |\ck|, 
\]
and
\[
\sum_{\bk \in \ck} \wt{x}^{r_n}_{\bk} = L(\wt{\bx}^{r_n}, \boldeta) \le r_n^{1-p} u^* + |\ck|,
\]
contradicting the fact that 
\[
\sum_{\bk \in \ck} \wt{x}^{r_n}_{\bk} - r_n^{1-p} u^* \ge 2|\ck|
\]
for sufficiently large $n$. This establishes Lemma \ref{lem:existence-si-2}.
\end{proof}

\noindent {\bf Proof of Proposition \ref{prop:opt-si-1}.} 
We are now ready to prove Proposition \ref{prop:opt-si-1}. 
Suppose that the proposition does not hold. 
Then for all $\veps > 0$, there exist infinitely many $r$ 
and $\bx^r$ such that $d(\bx^r, \cx^*) \ge 2D|\ck|r^{p-1}$, 
and for all SI pairs (if any) $\{(\bk', i), (\bk, i)\}$ of $\wt{\bx}^r$, 
either $\wt{x}^r_{\bk} < \veps$, or $\wt{x}^r_{\bk' - \be_i} < \veps$, 
or $\Delta_{(\bk', i)}(\wt{\bx}^r) - \Delta_{(\bk, i)}(\wt{\bx}^r) > -\veps$. 
Thus, we can find a subsequence $\{r_n\}$ of $\{r\}$ 
and states $\bx^{r_n}$ such that 
\begin{itemize}
\item[1.] $\bx^{r_n} \rightarrow \bx \in \cx$ as $n \rightarrow \infty$, 
\item[2.] $\wt{\bx}^{r_n} \rightarrow \wt{\bx}$ as $n \rightarrow \infty$, 
with some components of $\wt{\bx}$ being possibly infinite, 
\item[3.] $d(\bx^{r_n}, \cx^*) \ge 2D|\ck|r_n^{p-1}$ for all $n$, and 
\item[4.] for all SI pairs $\{(\bk', i), (\bk, i)\}$ associated with $\wt{\bx}^{r_n}$ (if any), 
either $\wt{x}^{r_n}_{\bk} < 1/n$, or $\wt{x}^{r_n}_{\bk' - \be_i} < 1/n$, 
or $\Delta_{(\bk', i)}(\wt{\bx}^{r_n}) - \Delta_{(\bk, i)}(\wt{\bx}^{r_n}) > -1/n$.
\end{itemize}
From Property $4$, we can deduce that $\wt{\bx}$ does not have an SI pair. 
But by Property $3$, this contradicts Lemma \ref{lem:existence-si-2}. 
This establishes Proposition \ref{prop:opt-si-1}. \hfill\(\Box\)

\subsection{Proof of Theorem \ref{thm:main-closed}}\label{ssec:proof-closed}

We will assume WLOG the following construction of the probability space.
For each $(\bk,i) \in \cm$, consider an independent
unit-rate Poisson process $\{\Pi_{(\bk, i)}(t), ~t\ge 0\}$. 
Assume that, for each $r$,
the Markov process $\bX^r(\cdot)$ is driven by this common
set of Poisson processes $\Pi_{(\bk, i)}(\cdot)$, as follows.
For each $(\bk, i)\in \cm$, let us denote by $D^r_{(\bk, i)}(t)$ the total 
number of type-$i$ service completions from servers of configuration $\bk$, 
 in the time interval $[0,t]$. 
 Then
\beql{eq-driving}
D^r_{(\bk, i)}(t) = \Pi_{(\bk, i)} \left(\int_0^t X_{\bk}^r(\xi) k_i \mu_i d\xi\right).
\end{equation}

\begin{lem}\label{lem:unif-conv}
Let $T > 0$ be fixed. With probability $1$, the following property holds. 
Consider any sequence $\{t_0^r\}_r$ with 
$t_0^r \in [0, Tr^{2-p}]$. Then for any $\xi \in [0, 1]$, 
and for any $(\bk, i) \in \cm$, 
\[
\frac{1}{r^{2p-1}}\left(\Pi_{(\bk, i)}\left(t_0^r + \xi r^{2p-1} \right) - \Pi_{(\bk, i)}\left(t_0^r\right)\right) 
\rightarrow \xi 
\]
as $r \rightarrow \infty$. The convergence is uniform over $t_0^r, \xi$, and $(\bk, i)$ 
in the following sense. 
For any $\veps > 0$, there exists $r(\veps)$ such that 
for all $r \geq r(\veps)$, $\xi \in [0, 1]$, $(\bk, i) \in \cm$, and $t_0^r \in [0, Tr^{2-p}]$, 
\[
\max_{(\bk, i), \xi, t_0^r}\left| \frac{1}{r^{2p-1}}\left(\Pi_{(\bk, i)}\left(t_0^r + \xi r^{2p-1} \right) - \Pi_{(\bk, i)}\left(t_0^r\right)\right) - \xi \right| < \veps. 
\]
\end{lem}
{The proof of Lemma \ref{lem:unif-conv} 
depends on simple large-deviation type estimates for Poisson random variables. 
The idea is essentially the same as that of Lemma 4.3 in \cite{SS2002}: 
we partition the interval $[0, Tr^{2p-1}]$ into subintervals of length $r^{p-1/2}$, 
and for each of them write the probability that the average increase rate 
of $\Pi_{(\bk, i)}$ lies outside $(1-\veps,1+\veps)$.
These probabilities are 
 $\exp\left(-\mbox{poly}(r)\right)$, 
and we only have $\mbox{poly}(r)$ such subintervals 
(here $\mbox{poly}(r)$ means a polynomial in $r$). This is true for any $\veps>0$.
We can then cover {\em any} subinterval of length $r^{2p-1}$ 
by these subintervals of length $r^{p-1/2}$.
We omit a detailed proof here.}

The following corollary is a simple consequence of Lemma \ref{lem:unif-conv}.
\begin{cor}\label{cor:del-x}
Let $T$ be fixed. With probability $1$, the following holds. 
For sufficiently large $r$, 
\beql{eq-bounded-rate}
\max_{\substack{ \xi \in [0, 1], \\ t^r_0 \in [0, Tr^{1-p}]}} d\left(\bX^r(t^r_0 + \xi r^{p-1}), \bX^r(t^r_0) \right) \leq 2 \bar \mu |\ck| r^p,
\end{equation}
where $\bar{\mu} = \max_{i\in\ci} \mu_i$, and $\mu_i$ is the service rate for type-$i$ customers.
\end{cor}
\begin{proof}
Consider the probability-$1$ event in Lemma \ref{lem:unif-conv}, 
in which we can and do replace $T$ with
$2 \bar \mu T$. 
(We do this because the total ``instantaneous'' rate of all transitions
is upper bounded by $2 \bar \mu r$.)
The rate of departure of type-$i$ customers is $\rho_i \mu_i r \le \rho_i \bar{\mu} r$, 
and the total rate of customer departure is no greater than $\sum_{i \in \ci} \rho_i \bar{\mu} r = \bar{\mu} r$.
Thus, for each $\bk \in \ck$, the rate of change in $X_{\bk}$ is at most $\bar{\mu} r$. 
For an interval of length $r^{p-1}$, the total change in $X_{\bk}$ 
is at most $O(r\cdot r^{p-1}) = O(r^p)$. 
More precisely, with probability $1$, for each $\bk \in \ck$, 
\[
\limsup_{r \rightarrow \infty} \frac{1}{r^p}
\max_{\substack{\xi \in [0, 1], \\t^r_0 \in [0, Tr^{1-p}]}} \left|X^r_{\bk}(t^r_0 + \xi r^{p-1}) - X^r_{\bk}(t^r_0) \right| \leq \bar{\mu}. 
\]
Thus, for sufficiently large $r$, and for each $\bk \in \ck$, 
\[
\max_{\substack{\xi \in [0, 1],\\t^r_0 \in [0, Tr^{1-p}]}} \left|X^r_{\bk}(t^r_0 + \xi r^{p-1}) - X^r_{\bk}(t^r_0) \right| \leq 2 \bar{\mu} r^p.
\]
Summing over the above expression establishes the corollary.
\end{proof}

\begin{prop}\label{prop:loc-decr}
There exist positive constants $C_1$ and $\delta$ 
such that the following holds. 
Let $T > 0$ be given. 
Then w.p.$1$, 
for all sufficiently large $r$, and for any interval 
$[t_0, t_0+r^{p-1}] \subset [0, Tr^{1-p}]$, 
if $d\left(\bx^r(t_0), \cx^*\right) \ge C_1 r^{p-1}$, then 
\[
F^r\big(\bX^r(t_0 + r^{p-1})\big) - F^r\big(\bX^r(t_0)\big) 
\leq -\delta r^{2p-1}.
\]
\end{prop}
\begin{proof}
The proof idea is as follows. Consider the increase in $F^r$ 
at each state transition. For concreteness, suppose that 
the current system state is $\bX^r$, and 
a type-$i$ customer just completed its service requirement 
on a server with configuration $\bk$, and is placed into a server 
with configuration $\bk'$. Then it is a simple calculation to see that 
the increase in $F^r$ is at most
\[
\Delta^r_{(\bk', i)}(\bX^r) - \Delta^r_{(\bk, i)}(\bX^r) + 4r^{-p}.
\]
The term $\Delta^r_{(\bk', i)}(\bX^r) - \Delta^r_{(\bk, i)}(\bX^r)$ 
captures the first-order increase in $F^r$, and 
the term $4r^{-p}$ bounds the second-order increase in $F^r$.
We will see that over an interval of length $r^{p-1}$, 
the increase in $F^r$ due to first-order terms is at most
$-O(r^{2p-1})$, and the increase due to second-order terms 
is at most a constant. We now proceed to the formal proof.

From now on, we work with the probability-$1$ event defined in Lemma \ref{lem:unif-conv}, 
under which 
\[
\frac{1}{r^{2p-1}}\left(\Pi_{(\bk, i)}\left(t_0 + \xi r^{2p-1} \right) - \Pi_{(\bk, i)}\left(t_0\right)\right) 
\rightarrow \xi 
\]
as $r \rightarrow \infty$, uniformly over $t_0, \xi$, and $(\bk, i)$. 
Let $C_1 = 2(\bar{\mu} + D)|\ck|$, where $\bar{\mu} = \max_{i\in \ci} \mu_i$ 
and $D$ is the same as in Lemma \ref{LEM:LP-RATE}.
Let $\varepsilon > 0$ be the same as in Proposition \ref{prop:opt-si-1}, 
and let $\delta > 0$ be such that $\delta < \frac{1}{8}\mu_i\veps^2$ for all $i \in \ci$. 

Claim that for all sufficiently large $r$, and for any interval 
$[t_0, t_0 + r^{p-1}] \subset [0, Tr^{1-p}]$, if $d\left(\bx^r(t_0), \cx^*\right) \ge C_1r^{p-1}$, 
then 
\[
F^r\big(\bX^r(t_0 + r^{p-1})\big) - F^r\big(\bX^r(t_0)\big) 
\leq -\delta r^{2p-1}.
\]

Suppose the contrary. Then there exist a subsequence of $\{r\}$ 
(which, with an abuse of notation, we still index by $r$), 
along which we have some $[t_0^{r}, t_0^{r} + r^{p-1}] \subset [0, Tr^{1-p}]$, 
such that $d\left(\bx^r(t_0^{r}), \cx^*\right) \geq C_1 r^{p-1}$, and 
\begin{equation}\label{eq:contrary-drift-F}
F^{r}\big(\bX^{r}(t_0^{r} + r^{p-1})\big) - F^{r}\big(\bX^{r}(t_0^{r})\big) 
> -\delta r^{2p-1}. 
\end{equation}
First, for sufficiently large $r$, and for all $\xi \in [0, 1]$, 
there exists a SI pair $\{(\bk', i), (\bk, i)\}$ associated with 
$\bx^{r}(t_0^{r} + \xi r^{p-1})$ (possibly depending on $r$ and $\xi$), 
such that 
\begin{align}
& \wt{x}^{r}_{\bk}(t_0^{r} + \xi r^{p-1}) \ge \veps, \quad
\wt{x}^{r}_{\bk' - \be_i}(t_0^{r} + \xi r^{p-1}) \ge \veps, \mbox{ and}\label{eq:si-pair-1}\\
& \Delta_{(\bk', i)}(\wt{\bx}^{r}(t_0^{r} + \xi r^{p-1})) 
- \Delta_{(\bk, i)}(\wt{\bx}^{r}(t_0^{r} + \xi r^{p-1})) \le -\veps. \label{eq:si-pair-2}
\end{align}
By Corollary \ref{cor:del-x}, for all $\xi \in [0, 1]$, 
$d\left(\bX^{r}(t_0^{r} + \xi r^{p-1}), \bX^r(t_0^r)\right) \leq 2\bar{\mu} |\ck| r^p$. 
Using triangle inequality and choosing $C_1 > 2(\bar{\mu} + D) |\ck|$, we have that 
for sufficiently large $r$, and for all $\xi \in [0, 1]$, 
\[
d\left(\bx^{r}(t_0^{r} + \xi r^{p-1}), \cx^*\right) \ge 2D|\ck| r^{p-1}.
\]
\eqref{eq:si-pair-1} and \eqref{eq:si-pair-2} now follow from Proposition \ref{prop:opt-si-1}. 

Fix a sufficiently large $r$ so that \eqref{eq:si-pair-1} and \eqref{eq:si-pair-2} hold. 
We then consider the first-order change in $F^r$ over the interval $[t_0^{r}, t_0^{r} + r^{p-1}]$ 
(i.e., the difference of $\Delta$). To do this, we partition $[t_0^{r}, t_0^{r} + r^{p-1}]$ 
into subintervals of length $c\veps r^{p-1}$, 
with $c > 0$ chosen small enough so that on each subinterval, 
there exists a {\em fixed} SI pair $\{(\bk', i), (\bk, i)\}$ such that 
\eqref{eq:si-pair-1} and \eqref{eq:si-pair-2} hold for this SI pair, 
and with $\veps$ replaced by $\veps/2$. We now argue that this can be done. 
Consider the first such subinterval, for example. 
By Lemma \ref{lem:unif-conv}, for sufficiently large $r$, 
the number of state transitions over this subinterval 
is at most $(c\veps r^{p-1})\cdot O(r) = O(\veps r^p) < \frac{1}{8}\veps r^p$, 
by choosing a sufficiently small $c$. 
This implies that for each $\bk \in \ck$, 
the change in $\tilde{x}^r_{\bk}$ over this subinterval 
is at most $\frac{1}{8}\veps$. Thus, \eqref{eq:si-pair-1} and \eqref{eq:si-pair-2} 
hold for an SI pair associated with $\tilde{\bx}^r (t_0^r)$, 
with $\veps$ replaced by $\veps/2$. The same argument holds  
for other subintervals.

Now concentrate on the subinterval $[t_0^r, t_0^r + c\veps r^{p-1}]$, 
and a corresponding SI pair $\{(\bk', i), (\bk, i)\}$ 
associated with $\tilde{\bx}^r(t_0^r)$ for which \eqref{eq:si-pair-1} and \eqref{eq:si-pair-2} 
hold on this subinterval with $\veps$ replaced by $\veps/2$. 
The number of type-$i$ departures from servers of configuration $\bk$ 
is at least $\mu_i \cdot \frac{\veps r^p}{2}\cdot (c\veps r^{p-1}) = \frac{1}{2} c\mu_i \veps^2 r^{2p-1}$.
At each such departure, 
the first-order increase (due to the difference of $\Delta$) in $F^{r}$ 
is at most $-\veps/2$, since {\em GSS} results in a smaller first-order increase than 
moving the departure to a server with configuration $\bk' - \be_i$. Summing over 
all such increases over type-$i$ departures gives a first-order increase in $F^r$ 
which is at most 
\[
-\frac{\veps}{2} \cdot \left(\frac{1}{2}c \mu_i \veps^2 r^{2p-1}\right) \le -2c \veps \delta r^{2p-1}.
\]
Exactly the same argument holds for other subintervals, so the total first-order increase in $F^r$ 
is at most $-2\delta r^{2p-1}$.

Finally, consider the second-order increase in $F^{r}$. 
As discussed at the beginning of the proof, 
the second-order increase in $F^{r}$ at each state transition 
is at most $4r^{-p}$. For sufficiently large $r$, 
the total number of state transitions 
over the interval $[t_0^{r}, t_0^{r} + r^{p-1}]$ 
is at most $r^{p-1}\cdot O(r) = O(r^p)$, 
and hence the total second-order increase in $F^{r}$ 
is at most $(4r^{-p}) \cdot O(r^p) = O(1)$.
Thus, for sufficiently large $r$, 
\[
F^{r}\big(\bX^{r}(t_0^{r} + r^{p-1})\big) - F^{r}\big(\bX^{r}(t_0^{r})\big) 
\le -2\delta r^{2p-1} + O(1) 
\le -\delta r^{2p-1}. 
\]
This contradicts \eqref{eq:contrary-drift-F}, and we have established the proposition.
\end{proof}

\begin{prop}\label{PROP:CONV-PROB} 
There exist positive constants $C$ and $T$ such that as $r \rightarrow \infty$, 
\[
\pr\left(d\left(\bx^r(Tr^{1-p}), \cx^*\right) \leq Cr^{p-1} \right) \rightarrow 1.
\]
\end{prop}
\begin{proof}[Sketch] 
{
The proof is very intuitive. 
We keep track of the evolution of $F^r$ on the interval $[0, Tr^{1-p}]$
subdivided into $r^{p-1}$-long subintervals. W.p.1., for all sufficiently large $r$,
the following is true for each subinterval $[t_0,t_0+r^{p-1}]$: $F^r$ decreases
by at least $\delta r^{2p-1}$ if $d\left(\bx^r(t_0), \cx^*\right) \ge C_1 r^{p-1}$
(by Proposition \ref{prop:loc-decr}), and it can never increase by more than
$C_3 r^p$. Therefore, if we choose $T$ large enough, then
$d\left(\bx^r(t), \cx^*\right) < C_1 r^{p-1}$ at some time $t\in[0,Tr^{1-p}]$
(because otherwise $F^r$ would become negative),
and $d\left(\bx^r(t), \cx^*\right)=O(r^{p-1})$ thereafter.
We refer the readers to Appendix \ref{apdx:conv-prob} 
for details.}
\end{proof}

\noindent {\bf Proof of Theorem \ref{thm:main-closed}.} 
Theorem \ref{thm:main-closed} is now a simple consequence of Proposition \ref{PROP:CONV-PROB}. 
For each $r$, 
consider $\bx^r(\cdot)$ in the stationary regime.
In particular, for any $T > 0$, $\bx^r(Tr^{1-p})$ has the same distribution as $\bx^r(\infty)$. 
Therefore, by Proposition \ref{PROP:CONV-PROB}, 
\[
\pr\left(d\left(\bx^r(\infty), \cx^*\right) \le C r^{p-1} \right) \rightarrow 1, 
\]
as $r \rightarrow \infty$. 
This completes the proof of Theorem \ref{thm:main-closed}. \hfill\(\Box\)

\section{Open System: Asymptotic \\Optimality of (Modified) {GSS}}
\label{sec-open}

{We prove Theorem \ref{th-main-res-open} in this section. 
The proof ``extends''  that of Theorem \ref{thm:main-closed}. 
The main additional step is Theorem \ref{thm:local-fluid-tightness}, which shows that in steady state, 
for each $i \in \ci$, $\tilde Y_i^r(t)$  
the number of tokens of type-$i$, remains $o(r^p)$ with high probability, over $O(r^{1-p})$-long intervals.}
As a starting point, we need the following facts.

\begin{thm}
\label{THM:SQRT-R-TIGHTNESS}
Consider the sequence (in $r$) of open systems in steady state. Consider any  fixed $i$.
There exists a positive constant $c$ such that, uniformly on all $r$, 
$$
\E \exp\{\|r^{-1/2} (\hat Y^r_i(\infty)-\rho_i r,\tilde Y^r_i(\infty))\|  \} \le c.
$$
\end{thm}
\begin{proof} See Appendix \ref{apdx:sqrt-r-tightness}.
\end{proof}
For our purposes, the following corollary will suffice.

\begin{cor}
\label{cor-tightness}
Consider the sequence (in $r$) of open systems in steady state. Consider any fixed $i$. Then, for any $q>1/2$,
$$
\|r^{-q} (\hat Y^r_i(\infty)-\rho_i r,\tilde Y^r_i(\infty))\| \Longrightarrow 0.
$$
\end{cor}

Next we show that the property of Corollary~\ref{cor-tightness} holds not just at a given time, but uniformly
on a $O(r^{1-q})$-long interval.

\begin{thm}
\label{thm:local-fluid-tightness}
Consider the sequence (in $r$) of open systems in stationary regime.
Consider any fixed $i$.
Let $q>1/2$ and $T>0$ be fixed. 
Then, as $r\to\infty$,
\beql{eq-local-fluid-tightness-weak}
\sup_{t\in[0,Tr^{1-q}]} \|r^{-q} (\hat Y^r_i(t)-\rho_i r,\tilde Y^r_i(t))\| \Longrightarrow 0,
\end{equation}
and, consequently,
\beql{eq-Z-approx-r-weak}
\sup_{t\in[0,Tr^{1-q}]} r^{-q} \|Z^r(t) -r\|\Longrightarrow 0.
\end{equation}
\end{thm}

Clearly, the statement of Theorem~\ref{thm:local-fluid-tightness} is equivalent to the following one:
{\em Any subsequence 
of $\{r\}$ contains a further subsequence along which w.p.1,
\beql{eq-local-fluid-tightness}
\sup_{t\in[0,Tr^{1-q}]} \|r^{-q} (\hat Y^r_i(t)-\rho_i r,\tilde Y^r_i(t))\|\to 0,
\end{equation}
and then
\beql{eq-Z-approx-r}
\sup_{t\in[0,Tr^{1-q}]} r^{-q} \|Z^r(t) -r\|\to 0.
\end{equation}
}
In turn, to prove the latter statement it suffices to show that  {\em there exists a construction of the underlying probability space,
for which the statement holds.} 

We will need some estimates, which can be obtained 
from a strong approximation of Poisson processes, 
available in, for example, \cite[Chapters 1 and 2]{Csorgo_Horvath}: 
\begin{prop}\label{thm:strong approximation-111-clean}
A unit rate Poisson process $\Pi(\cdot)$ and 
a standard Brownian motion $W(\cdot)$ can be constructed on a common
probability space in such a way that the following holds.
For some fixed positive constants $C_1$, $C_2$, $C_3$,   such that 
$\forall T>1$ and
$\forall u \geq 0$
\[
\pr\left(\sup_{0 \leq t \leq T} |\Pi(t) - t - W(t)| \geq C_1 \log T + u\right) \leq C_2 e^{-C_3 u}.
\]
\end{prop}
If in the above statement we replace $T$ with $rT$, and $u$ with $r^{1/4}$, we obtain
\begin{align}
& \pr\left(\sup_{0 \leq t \leq rT} |(\Pi(t) - t) - W(t)| < C_1 \log (rT) + r^{1/4}\right) \nonumber \\
& > 1- C_2 e^{-C_3 r^{1/4}}. \label{eq-add1}
\end{align}
Note also that for a fixed $\delta\in (0,q-1/2)$ and all large $r$,
\beql{eq-add2}
\pr\left(\sup_{0 \leq t \leq rT} |W(t)| \le r^{1/2+\delta}\right) \ge 1- e^{cr^{2\delta}}
\end{equation}
for some constant $c>0$. If events in \eqn{eq-add1} and  \eqn{eq-add2} hold for all large $r$, then
\beql{eq-add3}
\sup_{0 \leq t \leq rT} r^{-q}|\Pi(t)-t| \to 0.
\end{equation}

To prove Theorem~\ref{thm:local-fluid-tightness}, consider the following construction
of the probability space. (We want to strongly emphasize that this construction 
will be used only for the purpose of proving Theorem~\ref{thm:local-fluid-tightness}. 
For the proof of Theorem~\ref{th-main-res-open}, we can and will use
a different probability space
construction.)
For each $r$,
we divide the time interval 
$[0,Tr^{1-q}]$ into $r^{1-q}$ of  $T$-long subintervals, namely
$[(m-1)T,mT]$ with $m=1,2,\ldots,r^{1-q}$. In each of the subintervals,
and for each $r$, we consider independent
unit rate Poisson processes $\Pi_i^{r,m}$, $\hat \Pi_i^{r,m}$, $\tilde \Pi_i^{r,m}$,
driving type $i$ exogenous arrivals, actual customer departures
and token departures, respectively. 
More precisely, the number  of type $i$ exogenous arrivals, actual customer departures
and token departures, by time $t$ from the beginning of the $m$-th 
interval is given by
$$
\Pi_i^{r,m}(\lambda_i r t), ~~~ \hat \Pi_i^{r,m}\left(\int_0^t {\mu_i}\hat Y_i^r(\xi)d\xi\right), 
~~~\tilde \Pi_i^{r,m}\left(\int_0^t {\mu_0}\tilde Y_i^r(\xi)d\xi\right),
$$
respectively.
Using \eqn{eq-add1}-\eqn{eq-add3}
we obtain the following property for $\Pi_i^{r,m}$ (and analogous ones for $\hat \Pi_i^{r,m}$ and $\tilde \Pi_i^{r,m}$): 
\beql{eq-key-for-tightness-proof}
\max_{1 \le m \le r^{1-q}} ~~\max_{0\le t \le rT} |\Pi_i^{r,m}(t) - t|/r^{q}
\to 0, ~~\mbox{as $r\to\infty$, ~~w.p.1}.
\end{equation}

We denote 
$$
g^r(t) = (\hat y^r_i(t),\tilde y^r_i(t)) = r^{-q} (\hat Y^r_i(t)-\rho_i r,\tilde Y^r_i(t)).
$$
Then, we can prove the following.
\begin{lem}
\label{lem-local-fluid-conv}
Consider fixed realizations (for each $r$) of driving processes, such that 
the properties \eqn{eq-key-for-tightness-proof} hold 
with $q$ replaced by a smaller parameter $q'\in (1/2,q)$. 
Consider the corresponding sequence of realizations of $(g^r(t), ~t\ge 0)$,
with bounded initial states
$\|g^r(0)\|\le \epsilon$, $\epsilon>0$.
Then, there exists a subsequence of $r$ along which
\beql{eq-conv-to-local-fluid}
g^r(t) \to g(t), ~~~\mbox{u.o.c.},
\end{equation}
where $(g(t), ~t\ge 0)$ is Lipschitz continuous,
 with $\|g(0)\|\le\epsilon$, and it satisfies conditions
\beql{eq-yhat}
(d/dt) \hat y_i(t) = - \mu_i \hat y_i(t),
\end{equation}
\beql{eq-ytilde}
(d/dt) \tilde y_i(t) = \left\{ \begin{array}{ll}
                \mu_i \hat y_i(t) - \mu_0 \tilde y_i(t), 
                & \mbox{if}~\tilde y_i(t)>0\\
               \max\{0, \mu_i \hat y_i(t) - \mu_0 \tilde y_i(t)\}, 
                & \mbox{if}~\tilde y_i(t)=0
                \end{array}
       \right.
\end{equation}
at points $t\ge 0$, where the derivatives exist (which is almost everywhere
w.r.t. the Lebesgue mesure).
Moreover, the convergence
\beql{eq-y-converge}
\|g(t)\| \to 0, ~~t\to \infty,
\end{equation}
holds and is uniform w.r.t. initial states with $\|g(0)\|\le \epsilon$, and 
\beql{eq-y-bounded}
\sup_{\|g(0)\|\le \epsilon} \max_{t\ge 0} \|g(t)\| \to 0, ~~\epsilon \to 0.
\end{equation}
As a consequence of \eqn{eq-y-bounded},
\beql{eq-y-equilibrium}
\|g(0)\| = 0 ~~\mbox{implies}~~ 
\|g(t)\| = 0, ~\forall t.
\end{equation}
\end{lem}

Lemma~\ref{lem-local-fluid-conv} is analogous to Lemma 14 in \cite{St2012}, except that
the 
space scaling by $r^{-q}$ is applied, as opposed to the fluid scaling
by $r^{-1}$, and the number of actual customers $\hat Y^r_i(t)$ is centered
before scaling.
The proof is somewhat more involved -- the main issue is that 
(unlike for the fluid limit) 
the Lipschitz property of the limit is no longer automatic, because
the rates of arrivals and departures in the system are $O(r)$, while the space 
is only scaled down by $r^q$. (That is why we need to use properties
\eqn{eq-key-for-tightness-proof}, as opposed to simply a strong law of large numbers.)
However, this issue can be resolved as in, for example,
the proof of Theorem 23 in \cite{SY2012}. We omit a detailed proof. 
\newline\newline
{\bf Proof of Theorem \ref{thm:local-fluid-tightness}.} 
By Corollary~\ref{cor-tightness}, 
we can choose a subsequence of $r$ (increasing sufficiently fast) so that
$$
\|g^r(0)\|\to 0, ~~\mbox{w.p.1}.
$$
Then, we use the construction of the probability space specified above,
which guarantees that w.p.1
the properties \eqn{eq-key-for-tightness-proof} hold 
with $q$ replaced by a smaller parameter $q'\in (1/2,q)$ --
let us consider any element of the probability space for which  the properties \eqn{eq-key-for-tightness-proof} do hold.
We claim that, for this element, \eqn{eq-local-fluid-tightness} holds. 
Suppose not. 
Then, there exists $\epsilon>0$ and a further subsequence of $r$, along which
$\tau^r=\min\{t ~|~ \|g^r(t)\|>\epsilon\} \le  Tr^{1-q}$.
By Lemma~\ref{lem-local-fluid-conv}, we can and do choose time duration $T_1>0$
such that any limit trajectory $g(t)$ with $\|g(0)\|\le \epsilon$ satisfies
$\|g(T_1)\|\le \epsilon/2$. For each $r$, consider the trajectory
of $g^r$ on the time interval $[\tau^r-T_1,\tau^r]$. (Suppose for now that
$\tau^r \ge T_1$ for all sufficiently large $r$.) Then we can choose 
a further subsequence of $r$ along which 
$g^r(\tau^r-T_1 +t)\to g(t)$ uniformly for $t\in [0,T_1]$, for a limit function $g(t)$ as in 
Lemma~\ref{lem-local-fluid-conv}. But, this is impossible because then
$\|g^r(\tau^r)\| \to \|g(T_1)\| \le \epsilon/2$. The case when
$\tau^r <T_1$ for infinitely many $r$ is even simpler: we choose 
a further subsequence along which this is true, and consider the 
trajectories of $g^r$ on the fixed time interval $[0,T_1]$.
In this case any limit trajectory $g(t)$ described in Lemma~\ref{lem-local-fluid-conv}
stays at $0$ in the entire interval $[0,T_1]$, because $\|g(0)\|= \lim_r \|g^r(0)\|=0$.
This means that $\|g^r(\tau^r)\| \to 0$, again a contradiction.
\hfill\(\Box\)

From this point on, we assume the following structure of the probability space.
(It is different from the one used for the proof of Theorem~\ref{thm:local-fluid-tightness},
which, as we discussed, was  for that proof only.)
There are common (for all $r$) unit rate Poisson processes driving the system,
defined as follows.
For each $(\bk,i)\in \cm$ and $\hat \bk \le \bk$, consider independent
unit-rate Poisson process $\hat \Pi_{(\bk,\hat \bk), i}(t), ~t\ge 0$,
so that the number of actual type $i$ customer departures from configuration $(\bk,\hat \bk)$
in the interval $[0,t]$
is equal to $\hat \Pi_{(\bk,\hat \bk), i}\left(\int_0^t {\mu_i} \hat k_i X_{(\bk,\hat \bk)}^r(\xi)d\xi\right)$.
Similarly, consider independent
unit-rate Poisson process $\left\{\tilde \Pi_{(\bk,\hat \bk), i}(t), ~t\ge 0\right\}$, so that
the number of type $i$ token departures 
from configuration $(\bk,\hat \bk)$ due to their expiration, 
is equal to 
$\tilde \Pi_{(\bk,\hat \bk), i}\left(\int_0^t {\mu_0}(k_i-\hat k_i) X_{(\bk,\hat \bk)}^r(\xi)d\xi\right)$.
Finally, for each $i\in \ci$, let $\{\Pi_i(t), ~t\ge 0\}$ be an independent
unit-rate Poisson process, such that the number of exogenous type $i$ arrivals
in $[0,t]$ is equal to $\Pi_i(\lambda_i r t)$.
For a fixed parameter $T>0$, whose value will be chosen later,
each of the above Poisson processes satisfies 
Lemma~\ref{lem:unif-conv}, in which we can and do replace $T$ with
$2T [(\bar \mu \vee \mu_0)+ \sum_i \lambda_i]$. 
(We do this because we will ``work'' with system sample paths such that
$\sum_i \hat Y_i = \sum_i (\hat Y_i^r +\tilde Y_i^r) < 2r$,
and for these sample paths the total ``instantaneous'' rate of all transitions
is upper bounded by $2 r [(\bar \mu \vee \mu_0)+ \sum_i \lambda_i]$.)

Denote by $\tilde D^r_i(t_1,t_2)$ the number of type-$i$ token departures (due to their expirations),
and by $\hat A^{**,r}_i(t_1,t_2)$ the total number of exogenous type-$i$ arrivals (of actual customers)
that do {\em not} replace type-$i$ tokens, all in the interval $(t_1,t_2]$.
Also, denote $Y^r_i(t_1,t_2)=Y^r_i(t_2)-Y^r_i(t_1)$.

\begin{thm}
\label{thm:non-typical-transitions}
Consider the sequence (in $r$) of open systems in stationary regime. 
Let $T>0$ be fixed. Then, any subsequence 
of $r$ contains a further subsequence such that, w.p.1, the following holds:
\beql{eq-non-typical-departures}
\tilde D^r_i(t_0,t_0+r^{p-1})/[r^p r^{p-1}]\to 0,
\end{equation}
\beql{eq-non-typical-arrivals}
\hat A^{**,r}_i(t_0,t_0+r^{p-1})/[r^p r^{p-1}]\to 0,
\end{equation}
uniformly on all intervals $[t_0,t_0+r^{p-1}]\subset [0,Tr^{1-p}]$.
\end{thm}

\begin{proof} Indeed, by Theorem~\ref{thm:local-fluid-tightness},
we can and do choose a subsequence of $r$ along which 
\eqn{eq-local-fluid-tightness}-\eqn{eq-Z-approx-r} hold w.p.1. Then,
\eqn{eq-non-typical-departures} follows from 
\eqn{eq-local-fluid-tightness}, which states that the number of tokens
$\tilde Y^r_i(t)$ is uniformly $o(r^p)$, and from the construction 
of the token departure processes, with the corresponding driving
processes $\tilde \Pi_{(\bk,\hat \bk), i}$ satisfying 
Lemma~\ref{lem:unif-conv}.
From \eqn{eq-local-fluid-tightness}
we also have the uniform convergence
$$
Y^r_i(t_0,t_0+r^{p-1})/[r^p r^{p-1}]\to 0.
$$
But, this along with \eqn{eq-non-typical-departures} implies uniform convergence
\eqn{eq-non-typical-arrivals} as well, because we have the conservation law
$$
Y^r_i(t_0,t_0+r^{p-1}) = \hat A^{**,r}_i(t_0,t_0+r^{p-1}) - \tilde D^r_i(t_0,t_0+r^{p-1}).
$$
The theorem is then proved.
\end{proof}

\noindent {\bf Proof of Theorem \ref{th-main-res-open}.}
Consider the sequence of the system processes in stationary regime.
Consider a fixed $T>0$, chosen to be sufficiently large, as in 
Proposition~\ref{PROP:CONV-PROB}.
Consider any subsequence of $r$. Then,
we can and do choose a further subsequence of $r$ along which, w.p.1,
\eqn{eq-local-fluid-tightness}-\eqn{eq-Z-approx-r} hold with some $q\in (1/2,p)$
(by Theorem~\ref{thm:local-fluid-tightness}) ,
and the properties stated in Theorem~\ref{thm:non-typical-transitions} hold.
As in the proof of Proposition~\ref{PROP:CONV-PROB}, we will keep track of the evolution 
of the value of $F^r(\bX^r(t))$. We emphasize that this is exactly the same function
$F^r$ as defined in Section~\ref{sec-GSS-definition} and used in 
the analysis of closed system,
namely it has the fixed
parameter $r$ (in the system with index $r$), and {\em not} the random ``parameter'' $Z^r$.
We claim that the following property holds.

{\bf Claim:} {\em  
There exist positive constants $0 < C_1 < C_2$, $\delta>0$, such that the following holds.
For all 
sufficiently large $r$, uniformly on all intervals
$[t_0,t_0+r^{p-1}]\subset [0,Tr^{1-p}]$, we have 
(a) $F^r(\bX^r(t_0)) - r u^* \ge C_1 r^p$ implies 
\[F^r(\bX^r(t_0+r^{p-1})) - F^r(\bX^r(t_0)) \le - \delta r^{2p-1},\]
and (b) $F^r(\bX^r(t_0)) - r u^*  \le C_1 r^p$ implies 
\[\sup_{\xi\in [0,1]} F^r(\bX^r(t_0+\xi r^{p-1})) - r u^*  \le C_2 r^p.\]
}
Clearly, (b) is analogous to Corollary~\ref{cor:del-x}
for the closed system and 
is proved exactly same way, with $\bar \mu$ in \eqn{eq-bounded-rate}
 replaced by $\bar \mu \vee \mu_0$. Statement (a) 
is analogous to Proposition~\ref{prop:loc-decr} for the closed system, 
and we prove it below.
It is also clear that the claim, along with 
\eqn{eq-local-fluid-tightness}-\eqn{eq-Z-approx-r},
implies the theorem statement via the argument almost verbatim repeating
that in the proof of Proposition~\ref{PROP:CONV-PROB}.

It remains to prove (a). The proof is the
same as that of Proposition~\ref{prop:loc-decr}, except that we have to make additional 
estimates accounting for: (i) token departures due to their expiration and
actual customer arrivals that do not find tokens;
(ii) the fact that {\em GSS-M} uses weight function $\bar w^r=\bar w^r(X;Z^r)$, 
as opposed to function $w^r=w^r(X)$ (which has constant $r$ as a parameter,
instead of the random variable $Z^r$).
This is because, {\em if we would have only transitions associated with 
actual customer departures and actual customer arrivals replacing tokens,
and the assignment decisions would be based on weight $w^r$ as opposed to 
$\bar w^r$}, then exactly the same drift estimates as those in the proof of
Proposition~\ref{prop:loc-decr} would apply.
Note that in (i) we consider exactly  those transitions
for which we have properties \eqn{eq-non-typical-departures}-\eqn{eq-non-typical-arrivals}.
Therefore, in any interval $[t_0,t_0+r^{p-1}]$
the ``worst case'' possible increase in $F^r(\bX^r)$ due to such transitions
is $o(r^{2p-1})$. (We omit obvious epsilon/delta formalities.)
Now consider (ii). Since we have the uniform bound $|Z^r(t)-r|\le O(r^q)$,
it is easy to check that $|\bar w^r(X) - w^r(X)|\le O(r^{q-1})$ for any $X\ge 0$.
This means that the error in the calculation of first-order contribution
into the change of $F^r(\bX^r)$ in any $[t_0,t_0+r^{p-1}]$, introduced
by {\em GSS-M} using weight $\bar w^r$ instead of $w^r$, is uniformly bounded
by $O(r r^{p-1} r^{q-1})=O(r^{p+q-1})= o(r^{2p-1})$. (Again, we omit epsilon/delta formalities.)
We see that the potential positive contribution of both (i) and (ii) into 
the change of objective function in any interval $[t_0,t_0+r^{p-1}]$
is $o(r^{2p-1})$, uniformly on the choice of the interval.
The estimate in (a) follows. Thus, the proof of the above claim, and 
of the theorem, follows.
\hfill\(\Box\)

\section{Discussion}\label{sec-discussion}
We presented the policy {\em Greedy with sublinear Safety Stocks (GSS)} along with a variant, 
which asymptotically minimize the steady-state total number of occupied servers at the fluid scale, 
as the input flow rates grow to infinity. 
A technical novelty of {\em GSS} is that it \emph{automatically} creates non-zero safety stocks,
{\em sublinear in the system ``size''}, at server configurations which have zero stocks on the fluid scale.
It is important to note that the algorithm does it 
without {\em a priori} knowledge of system parameters.
To prove the fluid-scale optimality 
of {\em GSS}, we also need to consider a local fluid scaling, 
under which the sublinear safety stocks are ``visible''.
This in turn allows us to obtain a tight asymptotic characterization of the algorithm deviation from 
exact optimal packing.

We can extend {\em GSS} to policies that asymptotically minimize the more general objective $\sum_{\bk} c_{\bk} X_{\bk}$, where $c_{\bk} > 0$ can be interpreted as the ``cost'' (for example, some estimated energy cost) of keeping a server in configuration $\bk$, for each $\bk \in \ck$. 
Instead of the weight function $w^r(X^r_{\bk})$ for each $\bk \in \ck$, 
consider the weight function $c_{\bk} w^r(X^r_{\bk})$, and define $\Delta^r$ 
as the difference between the new weight functions. 
We can then define {\em GSS} and {\em GSS-M} using the new $\Delta^r$. 
They minimize the fluid scale quantity $\sum_{\bk} c_{\bk} x_{\bk}$ asymptotically, 
and similar convergence rates can be obtained. 
If we assume that the cost $c_{\bk}$ is monotonically non-decreasing in $\bk$ 
(i.e., $c_{\bk'} \le c_{\bk}$ if $\bk' \le \bk$), then all our results and proofs still hold essentially verbatim.
If costs $c_{\bk}$ are not monotone in $\bk$, 
most of the statements and proofs easily extend, 
except those of Lemmas \ref{lem:existence-si-1} and \ref{lem:existence-si-2}, 
where some dual variables $\eta_i$ may need to be negative. 
These $\eta_i$ can be defined in a similar fashion as those in the proof 
of Lemma 6 in \cite{St2012}.

There are some possible directions for future research. 
For example, one may expect asymptotic optimality of ``pure'' {\em GSS} in an open system, 
which seems more difficult to establish. 
Proving or disproving its optimality may require 
better understanding of and some new insight into the system dynamics. 
Another direction can be the investigation of policies other (possibly simpler) than {\em GSS}. 
{\em GSS} is {\em asymptotically} optimal as the system scale increases. 
However, if the number $|\ck|$ of feasible configurations is large, the system scale may need to be
very large for the near optimal performance. 
It is then of interest to design policies (e.g., some form of best-fit) 
that have provably good performance properties at a wide range of system scales.

\newpage
\appendix

\section{Proof of Lemma \ref{LEM:LP-RATE}}\label{apdx:lp-rate}
Both $\cx$ and $\cx^*$ are convex and compact polytopes 
with a finite number of extreme points. Let $\cs$ and $\cs^*$ be the set 
of extreme points of $\cx$ and of $\cx^*$, respectively. 
Note that for all $\bx^* \in \cs^*$, $\sum_{\bk} x^*_{\bk} = u^*$, 
and for all $\bx' \in \cs \backslash \cs^*$, $\sum_{\bk} x'_{\bk} > u^*+\delta$. 
for some $\delta > 0$.

Let $\hull{\cs\backslash \cs^*}$ be the convex hull of the set $\cs\backslash \cs^*$. 
Then for all $\bx' \in \hull{\cs\backslash \cs^*}$, $\sum_{\bk} x'_{\bk} \ge u^* + \delta$. 
Consider the function $g : \hull{\cs\backslash \cs^*} \times \cx^* \rightarrow \R$ 
defined by 
$g(\bx^*, \bx') = \|\bx^* - \bx'\|/\left(\sum_{\bk \in \ck} x'_{\bk} - u^*\right)$.
Function $g$ is well-defined, always positive and clearly continuous. 
Since both $\hull{\cs\backslash \cs^*}$ and $\cx^*$ are compact, so is their product space. 
Thus there exists $D > 0$ such that $g$ is upper bounded by $D$. 

For every $\bx \in \cx$, there exists $\lambda \in [0, 1]$ 
such that $\bx = \lambda \bx' + (1-\lambda) \bx^*$, with $\bx' \in \hull{\cs\backslash \cs^*}$ 
and $\bx^* \in \cx^*$. Then 
\begin{eqnarray*}
d\left(\bx, \cx^*\right) &\le& \|\bx - \bx^*\| = \lambda \|\bx' - \bx^*\| \\
& = & \lambda g(\bx^*, \bx') \left(\sum_{\bk \in \ck} x'_{\bk} - u^*\right) \\
& \le & \lambda D\left(\sum_{\bk \in \ck} x'_{\bk} - u^*\right) 
 = D\left(\sum_{\bk \in \ck} x_{\bk} - u^*\right).
\end{eqnarray*}

\section{Proof of Proposition \ref{PROP:CONV-PROB}}\label{apdx:conv-prob}
Let $\delta > 0$ be the same as in Proposition \ref{prop:loc-decr}, 
and define $T = 3/\delta$. 
{$C > 0$ will be chosen to be sufficiently large, 
whose value will be determined later in the proof.}
Clearly, to prove the proposition, it suffices to 
prove a stronger property
\[
\pr\left(d\left(\bx^{r}(Tr^{1-p}), \cx^*\right) \leq Cr^{p-1} \mbox{ for all large $r$} \right) 
= 1.
\]
By Proposition \ref{prop:loc-decr}, there exists $C_1 > 0$ 
such that w.p.$1$, 
for sufficiently large $r$, and for any interval 
$[t_0, t_0 + {r}^{p-1}] \subset [0, T{r}^{1-p}]$, 
if $d\left(\bx^{r}(t_0), \cx^*\right) \ge C_1 r^{p-1}$, 
then
\begin{equation}\label{eq:local-drift-F}
F^r\big(\bX^r(t_0^r + r^{p-1})\big) - F^r\big(\bX^r(t_0^r)\big) \leq -\delta r^{2p-1}. 
\end{equation}
We pick some $r$ such that the above statement holds, 
and that furthermore, for every $t_0 \in [0, Tr^{1-p}]$ and $\xi \in [0, 1]$, 
\begin{equation}\label{eq:local-change-X}
d\left(\bX^r(t^r_0 + \xi r^{p-1}), \bX^r(t^r_0) \right) \leq O(r^p).
\end{equation}
This can be done by Corollary \ref{cor:del-x}.

Now claim that $d(\bx^r(Tr^{1-p}), \cx^*) \le C r^{p-1}$.
To establish the claim, 
we consider the set $\cl = \{\ell \in \Z_+: \ell r^{p-1} \in [0, T r^{1-p}]\}$, 
and prove that
\begin{itemize}
\item[(a)] there exists $\ell_0 \in \cl$ 
such that 
$d(\bx^r(\ell_0 r^{p-1}), \cx^*) \le C_1 r^{p-1}$, 
and
\item[(b)] there exists $C_2 > 0$ such that 
for all $\ell \in \cl$ with $\ell \ge \ell_0$, 
$F^r\big(\bX^r(\ell r^{p-1})\big) \le r u^* + C_2 r^p$.
\end{itemize}
First suppose that (a) does not hold. 
Then for every $\ell \in \cl$, $d(\bx^r(\ell r^{p-1}), \cx^*) \ge C_1 r^{p-1}$, so
\[
F^r\big(\bX^r((\ell+1)r^{p-1})\big) - F^r\big(\bX^r(\ell r^{p-1})\big) \leq -\delta r^{2p-1}.
\]
Let $\bar{\ell} = \lceil Tr^{2(1-p)} \rceil$. Summing these inequalities over $\ell$, 
we obtain
\begin{align*}
& F^r\big(\bX^r(\bar{\ell} r^{p-1})\big) - F^r\big(\bX^r(0)\big) \leq - \bar{\ell} \delta r^{2p-1} \\
& \le - (Tr^{2(1-p)}-1)\delta r^{2p-1} = -T\delta r + \delta r^{2p-1}.
\end{align*}
Thus,
\begin{align*}
& F^r\big(\bX^r(\bar{\ell} r^{p-1})\big) \le F^r\big(\bX^r(0)\big) - T\delta r + \delta r^{2p-1} \\
& \le r - \frac{3}{\delta} \delta r + \delta r^{2p-1} < 0.
\end{align*}
This contradicts the nonnegativity of $F^r$, so statement (a) is established. 

To establish statement (b), 
we use the following simple lemma, whose proof is omitted. 
\begin{lem}\label{lem:max-seq}
Let $K, \alpha$ and $\beta$ be given positive constants. 
Consider a sequence of real numbers $\{a_n\}$ that satisfies:
(i) $a_0 \le K$, (ii) $a_{n+1} - a_n \le \alpha$, and 
(iii) if $a_n \ge K$, then $a_{n+1} - a_n \le -\beta$. 
Then $\max_n a_n \le K + \alpha$.
\end{lem}
We will establish the following corresponding statements:

(i) $F^r \left(\bX^r(\ell_0 r^{p-1})\right) \le r u^* + C_1 r^p$. 
Recall that we have $d\left(\bx^r(\ell_0 r^{p-1}), \cx^*\right) \le C_1 r^{p-1}$, 
so by Lemma \ref{lem:F-sum-close}, 
\begin{align*}
& F^r \left(\bX^r(\ell_0 r^{p-1})\right) - ru^* \le \sum_{\bk \in \ck} X^r_{\bk} (\ell_0 r^{p-1}) - r u^* \\
& \le r d\left(\bx^r(\ell_0 r^{p-1}), \cx^*\right) \le C_1 r^p.
\end{align*}

(ii) There exists $C_3 > 0$ such that 
$F^r \left(\bX^r((\ell+1) r^{p-1})\right) - F^r \left(\bX^r(\ell r^{p-1})\right) 
\le C_3 r^p$. This is clear, since by Lemma \ref{lem:F-sum-close}, 
$F^r\left(\bX^r\right)$ differs from $\sum_{\bk} X^r_{\bk}$ by $O(r^p)$, 
and the change in $\bX^r$ is at most $O(r^p)$ over an interval of length $r^{1-p}$. 

(iii) If $F^r \left(\bX^r(\ell r^{p-1})\right) \ge r u^* + C_1 r^p$, 
then 
\[F^r\big(\bX^r((\ell+1)r^{p-1})\big) - F^r\big(\bX^r(\ell r^{p-1})\big) \leq -\delta r^{2p-1}.\]
To see this, suppose that $F^r \left(\bX^r(\ell r^{p-1})\right) \ge r u^* + C_1 r^p$. Then
$d\left(\bx^r(\ell r^{p-1}), \cx^*\right) \ge \sum_{\bk \in \ck} x^r_{\bk} - u^* 
 \ge \frac{1}{r} F^r \left(\bX^r(\ell r^{p-1})\right) - u^* \ge C_1 r^{p-1}$, 
and we must have 
\[
F^r\big(\bX^r((\ell+1)r^{p-1})\big) - F^r\big(\bX^r(\ell r^{p-1})\big) \leq -\delta r^{2p-1}.
\]

By Lemma \ref{lem:max-seq}, for all $\ell \in \cl$ with $\ell \ge \ell_0$, we have 
\[
F^r\big(\bX^r(\ell r^{p-1})\big) \le r u^* + \left(C_1 + C_3\right) r^p = r u^* + C_2 r^p,
\]
by letting $C_2 = C_1 + C_3$.
This establishes statement (b). In particular, for $\bar{\ell} = \lceil T r^{2(1-p)}\rceil$, 
\[
F^r\big(\bX^r(\bar{\ell} r^{p-1})\big) \le r u^* + C_2 r^p.
\]
Now by \eqref{eq:local-change-X}, the difference between 
$\bX^r(T r^{1-p})$ and $\bX^r(\bar{\ell}r^{p-1})$ is $O(r^p)$. 
Furthermore, the difference between $F^r\big(\bX^r(\bar{\ell} r^{p-1})\big)$ 
and $\bX^r(\bar{\ell}r^{p-1})$ also $O(r^p)$. This implies that 
\[
 \sum_{\bk \in \ck} X^r_{\bk}(T r^{1-p}) - r u^* \le C_2 r^{p} + O(r^p).
\]
Thus, there exists $C > 0$ 
such that 
\[
 \sum_{\bk \in \ck} x^r_{\bk}(T r^{1-p}) - u^* \le  \frac{C}{D} r^{p-1}.
\]
By Lemma \ref{LEM:LP-RATE}, 
\begin{align*}
& ~d(\bx^r(T r^{1-p}), \cx^*) \le D\left(\sum_{\bk \in \ck} x^r_{\bk}(T r^{1-p}) - u^*\right) \\
\le & ~D \cdot \frac{C}{D} r^{1-p} = C r^{1-p},
\end{align*}
and we have established the claim.
Therefore, w.p.$1$, 
\[
d(\bx^{r}(T r^{1-p}), \cx^*) \le Cr^{1-p},
\]
for all sufficiently large $r$.
This establishes the proposition.

\section{Proof of Theorem \ref{THM:SQRT-R-TIGHTNESS}}\label{apdx:sqrt-r-tightness}
The general approach of the proof is similar to that of Theorem 2 (ii) in \cite{GS2012}, 
in that it is based on the process generator estimates for the exponent $e^{\Phi}$, where $\Phi$ is
a function on the state space. However, the function $\Phi$ in our case is much different, and so are the specifics
of the estimates.
Consider fixed $i \in \ci$ and $r$. 
For notational convenience, we drop the subscript $i$ and superscript $r$ 
from all quantities considered in this proof. 
The Markov chain $\bU(\cdot) = (\hat Y(\cdot), \tilde Y(\cdot))$ has infinitesimal 
transition rate matrix $\bxi$ given by 
\[
\xi(\bu, \bu + \bv) \rightarrow \left\{ \begin{array}{ll}
\lambda r, & \mbox{ if } \bv = (1, -1\cdot\bOne_{\{\tilde y > 0\}}),\\
\mu \hat{y}, & \mbox{ if } \bv = (-1, 1), \\
\mu_0 \tilde{y}, & \mbox{ if } \bv = (0, -1), \\
0, & \mbox{ otherwise,}
\end{array}\right.
\]
where $\bu = (\hat y, \tilde y)$. We consider $A$ the infinitesimal generator 
of the Markov chain $\bU(\cdot)$, defined by 
\begin{equation}\label{eq:generator-def}
A G(\bu) = \sum_{\bu'} \xi(\bu, \bu')\left(G(\bu') - G(\bu)\right),
\end{equation}
for all functions $G : \Z_+^2 \to \R$ in the domain of $A$. 
We also consider the formal operator $\bar{A}$, 
defined (similar to Eq. \eqref{eq:generator-def}) by 
\begin{equation}\label{eq:for-gen-def}
\bar{A} G(\bu) = \sum_{\bu'} \xi(\bu, \bu')\left(G(\bu') - G(\bu)\right),
\end{equation}
for all functions $G: \Z_+^2 \to \R$.
Similarly to
\cite{GS2012}, it is easy to observe that 
the following property holds:
if a function $G$ takes a fixed constant value on the entire state space, except maybe a finite subset, then
$G$ is within the domain of $A$, 
$A G = \bar A G$, and moreover
\beql{eq-generators}
\E[AG(\bU)] = \E[\bar AG(\bU)] = 0,
\end{equation}
where the expectation is taken w.r.t 
the stationary distribution of the Markov chain $\bU(\cdot)$.  

First, define the (candidate) Lyapunov function $G : \Z_+^2 \to \R$ 
by 
\[
G(\bu) = \exp \left(\frac{1}{\sqrt{r}}h(\bu)\right),
\]
where $h(\bu) = \sqrt{(\hat y - \rho r)^2 + \frac{\mu_0}{\mu} \tilde y^2}$.
Note that, for an arbitrary $b\ge 0$, the truncated function 
$$
\bar G^{(b)}(\bu) = \exp \left(\frac{h(\bu)}{\sqrt{r}} \wedge b \right)
$$
is constant outside a finite subset and therefore, by \eqn{eq-generators}, 
\beql{eq-generators2}
\E[\bar AG^{(b)}(\bU)] = 0.
\end{equation}
Also note that, 
$$
\bar{A}G^{(b)}(\bu) \le \bar{A}G(\bu), ~~\mbox{if}~h(\bu)/\sqrt{r} \le b,
$$
$$
\bar{A}G^{(b)}(\bu) \le 0, ~~\mbox{if}~h(\bu)/\sqrt{r} \ge b.
$$

Similar to \cite{GS2012}, the following inequality can be derived, 
using Taylor expansion. 
There exists some constant $c_2 > 0$ such that for sufficiently large $r$, 
\begin{equation}\label{eq:drift-G-1}
\bar{A}G(\bu) \le G(\bu)\left(\frac{1}{\sqrt{r}}\bar{A}h(\bu) + \frac{c_2}{r}(\lambda r + \mu \hat y + \mu_0 \tilde y)\right).
\end{equation}
The term $\frac{G(\bu)}{\sqrt{r}}\bar{A}h(\bu)$ captures the first-order change in $G(\bu)$, 
and $\frac{c_2 G(\bu)}{r}(\lambda r + \mu \hat y + \mu_0 \tilde y)$ bounds the second-order change. 
Here we used the fact that $h$ is Lipschitz continuous and $\|\bu\|$ is changed by at most $1$ by any single transition.
Now consider the term $\bar{A}h(\bu)$. 
We use the following inequality to bound $\bar{A}h(\bu)$:
\[
\sqrt{(x+a)^2 + (y+b)^2} - \sqrt{x^2 + y^2} \leq \frac{ax + by + a^2 + b^2}{\sqrt{x^2 + y^2}}.
\]
To verify this inequality, note that first, 
\[
\left(\sqrt{(x+a)^2 + (y+b)^2}\right)^2 \leq 
\left(\sqrt{x^2 + y^2} + \frac{ax + by + a^2 + b^2}{\sqrt{x^2 + y^2}}\right)^2,
\]
and second, 
\[
\sqrt{x^2 + y^2} + \frac{ax + by + a^2 + b^2}{\sqrt{x^2 + y^2}} \geq 0.
\]
Thus, 
\begin{align}
\bar{A}h(\bu) \le & 
~\frac{(\lambda r - \mu \hat{y})(\hat{y} - \rho r) - (\lambda r - \mu \hat{y} + \mu_0 \tilde y)(\mu_0\tilde{y}/\mu)}{\sqrt{(\hat{y} - \rho r)^2 + \mu_0 \tilde{y}^2/\mu}} \nonumber \\
 & ~+ \frac{c_3(\lambda r + \mu \hat y + \mu_0 \tilde y)}{\sqrt{(\hat{y} - \rho r)^2 + \mu_0 \tilde{y}^2/\mu}} \nonumber \\
= & ~\frac{-\frac{\mu}{2}(\hat y - \rho r)^2 - \frac{\mu_0^2}{2\mu}\tilde{y}^2 - \frac{\mu}{2}(\hat y - \rho r + \tilde y)^2}{\sqrt{(\hat{y} - \rho r)^2 + \mu_0 \tilde{y}^2/\mu}} \nonumber \\
& ~+ \frac{c_3(\lambda r + \mu \hat y + \mu_0 \tilde y)}{\sqrt{(\hat{y} - \rho r)^2 + \mu_0 \tilde{y}^2/\mu}} \nonumber \\
\le & ~\frac{-\frac{\mu}{2}(\hat y - \rho r)^2 - \frac{\mu_0^2}{2\mu}\tilde{y}^2}{h(\bu)}
 + \frac{c_3(\lambda r + \mu \hat y + \mu_0 \tilde y)}{h(\bu)} \nonumber \\
\le & ~-c_4h(\bu) + \frac{c_3}{\sqrt{r}}(\lambda r + \mu \hat y + \mu_0 \tilde y), \label{eq:drift-G-2}
\end{align}
for  some positive constants $c_3$ and $c_4$, and when $h(\bu) \ge \sqrt{r}$. 
Combining Inequalities \eqref{eq:drift-G-1} and \eqref{eq:drift-G-2}, we have
\[
\bar{A}G(\bu) \le G(\bu)\left(-\frac{c_4}{\sqrt{r}} h(\bu) + \frac{c_2 + c_3}{r}(\lambda r + \mu \hat y + \mu_0 \tilde y)\right).
\]
Consider the term in the bracket on the RHS. It is now an elementary calculation to see that there exists some positive constant $c_5$, 
such that whenever $h(\bu) \ge c_5 \sqrt{r}$, 
\[
-\frac{c_4}{\sqrt{r}} h(\bu) + \frac{c_2 + c_3}{r}(\lambda r + \mu \hat y + \mu_0 \tilde y) \le -1.
\]
Also note that when $h(\bu) < c_5 \sqrt{r}$, 
the maximum values of 
\[
G(\bu)~~~\mbox{ and }~~~G(\bu)\left(-\frac{c_4}{\sqrt{r}} h(\bu) + \frac{c_2 + c_3}{r}(\lambda r + \mu \hat y + \mu_0 \tilde y)\right)
\]
are both bounded above by an absolute constant, say $c_6$, 
which does not depend on $r$. In summary, 
\begin{align*}
& \bar{A}G(\bu) \le -G(\bu) ~~\mbox{ whenever } ~~h(\bu) \ge c_5 \sqrt{r}, \\
\mbox{and }~~& \bar{A}G(\bu) \le c_6 ~~\mbox{ whenever } ~~h(\bu) < c_5 \sqrt{r}.
\end{align*}
Thus, 
for any $b>c_5$,
\begin{eqnarray*}
0 = \E[AG^{(b)}(\bU)] & \le & \E[\bar AG(\bU) \bOne_{\{c_5 \sqrt{r} \le h(\bU) \le b \sqrt{r}\}}] \\
& & + \E[\bar AG(\bU) \bOne_{\{h(\bU) < c_5 \sqrt{r}\}}] \\
& \le & - \E[G(\bU)\bOne_{\{c_5 \sqrt{r} \le h(\bU) \le b \sqrt{r}\}}] + c_6.
\end{eqnarray*}
This implies that $\E[G(\bU)\bOne_{\{c_5 \sqrt{r} \le h(\bU) \le b \sqrt{r}\}}] \le c_6$, and 
then $\E[G(\bU)\bOne_{\{ h(\bU) \le b \sqrt{r}\}}] \le 2 c_6$.
Finally, by Monotone Convergence, $\E[G(\bU)] \le 2c_6$. This completes the proof.

\end{document}